\documentclass[12pt]{amsart}

\usepackage{amsfonts, amstext, amsmath, amsthm, amscd, amssymb, enumitem, url}
\usepackage{tikz-cd}
\usepackage{xcolor}
\usepackage{svg}
\usepackage{pdfpages}
\usepackage{pdflscape}

\usepackage[parfill]{parskip}
\usepackage{microtype}

\usepackage[colorlinks,citecolor=cyan,linkcolor=magenta]{hyperref}
\usepackage[nameinlink]{cleveref}
\usepackage{multirow}
\usepackage{subfig}

\newtheorem{thm}{Theorem}[section]
\newtheorem{cor}[thm]{Corollary}
\newtheorem{lemma}[thm]{Lemma}
\newtheorem{prop}[thm]{Proposition}
\newtheorem{claim}[thm]{Claim}
\newtheorem{conj}[thm]{Conjecture}

\theoremstyle{definition}
\newtheorem{defn}[thm]{Definition}
\newtheorem{rmk}[thm]{Remark}

\numberwithin{equation}{section}

\newcommand{\RR}{\mathbb{R}}

\newcommand{\QQ}{\mathbb{Q}}
\newcommand{\ZZ}{\mathbb{Z}}
\renewcommand{\epsilon}{\varepsilon}

\newcommand{\spn}{\mathrm{span}}

\newcommand{\rad}{\mathrm{rad}}

\newcommand{\real}{\mathrm{real}}
\newcommand{\infs}{\mathrm{inf}}

\allowdisplaybreaks

\begin{document}

\title[Stretch factors of orientation-reversing pseudo-Anosov maps]{Minimal stretch factors of orientation-reversing fully-punctured pseudo-Anosov maps}

\author{Erwan Lanneau}
\address{
UMR CNRS 5582,
Univ. Grenoble Alpes, CNRS, Institut Fourier, F-38000 Grenoble, France}
\email{erwan.lanneau@univ-grenoble-alpes.fr}

\author{Livio Liechti}
\address{Department of Mathematics, University of Fribourg, Chemin du Musée 23, 1700 Fribourg, Switzerland}
\email{livio.liechti@unifr.ch}

\author{Chi Cheuk Tsang}
\address{Département de mathématiques \\
Université du Québec à Montréal \\
201 President Kennedy Avenue \\
Montréal, QC, Canada H2X 3Y7}
\email{tsang.chi\_cheuk@uqam.ca}

\begin{abstract}
We show that the stretch factor $\lambda(f)$ of an orientation-reversing fully-punctured pseudo-Anosov map $f$ on a finite-type orientable surface $S$, with $-\chi(S) \geq 4$ and having at least two puncture orbits, satisfies the inequality $\lambda(f)^{-\chi(S)} \geq \sigma^2$, where $\sigma=1+\sqrt{2}$ is the silver ratio. 
We provide examples showing that this bound is asymptotically sharp. This extends previous results of Hironaka and Tsang to orientation-reversing maps. 
\end{abstract}

\maketitle

\section{Introduction} \label{sec:intro}

It is an open and notoriously difficult problem to find, on a given closed orientable surface of genus~$g\ge3$, the minimal stretch factor among all orientation-preserving pseudo-Anosov maps of the surface. This problem is equivalent to finding the length of the shortest closed geodesic on the moduli space of Riemann surfaces of genus~$g$, with respect to the Teichmüller metric. For a survey of results, we refer to the introduction of~\cite{LS20}.

Recently, Hironaka and Tsang significantly advanced our understanding of the analogous problem in the fully-punctured case, that is, when the singularities of the invariant 
foliations of $f$ lies on the punctures of $S$, under the additional hypothesis that there are at least two puncture orbits. More precisely, for all even Euler characteristics, they determined the precise minimal stretch factor, and showed that for all Euler characteristics, the normalized stretch factor is always bounded from below by~$\mu^4\approx 6.854$, where~$\mu$ is the golden ratio~\cite[Theorem 1.1]{HT22}. 

Furthermore, Tsang determined the shape of the set of normalized stretch factors of fully-punctured orientation-preserving pseudo-Anosov maps~\cite[Theorem 1.6]{Tsa23}: it is the union of six isolated points and a dense subset of~$[\mu^4, \infty)$. 

In this article, we explore analogous questions for orientation-reversing pseudo-Anosov maps. Let $\sigma=1+\sqrt{2}$ be the silver ratio. Our first result is the following.

\begin{thm} \label{thm:mainthm}
Let~$f:S \to S$ be an orientation-reversing fully-punctured pseudo-Anosov map on a finite-type orientable surface. Suppose $-\chi(S) \geq 4$ and $f$ has at least two puncture orbits, then the normalized stretch factor of~$f$ satisfies
$$\lambda(f)^{-\chi(S)} \geq \sigma^2 \approx 5.828.$$
Moreover, this inequality is asymptotically sharp, in the sense that for every integer~$k \geq 2$, there exists an orientation-reversing fully-punctured pseudo-Anosov map~$f_k:S_k \to S_k$, where~$\chi(S_k)=-2k$, such that~$\lim_{k \to \infty} \lambda(f_k)^{-\chi(S_k)} = \sigma^2$.
\end{thm}

Sharpness in ~\Cref{thm:mainthm} follows from explicit examples, see~\Cref{prop:sharpnesseg}. For all even Euler characteristics~$-2k$,~$k\ge2$, we construct pseudo-Anosov maps that have as their stretch factor the minimal spectral radius among all skew-reciprocal primitive matrices~$A\in\mathrm{GL}_{2k}(\ZZ)$.
For~$k\ne3$, this stretch factor is smaller than the lower bound for the orientation-preserving case determined by Hironaka and Tsang~\cite[Theorem~1.9]{HT22}. We obtain the following corollary, comparing the stretch factors of orientation-preserving and orientation-reversing fully-punctured pseudo-Anosov maps on surfaces with even Euler characteristic. 

\begin{cor}\label{cor:maincor}
Fix any positive integer~$k\ge1$,~$k\ne3$. Among all fully-punctured pseudo-Anosov maps with at least two puncture orbits on surfaces with even Euler characteristic~$-2k$, the orientation-reversing ones realize a smaller stretch factor than the orientation-preserving ones. 
\end{cor}

The only case of \Cref{cor:maincor} not covered by \Cref{prop:sharpnesseg} is the case~$k=1$. However, in this case the result follows from a direct comparison: as per \cite{HT22}, the smallest normalized stretch factor in the orientation-preserving case is~$\mu^4$, while there exists an orientation-reversing map on the four-punctured sphere with normalized stretch factor~$\mu^2$, see \Cref{rmk:2A_1lowdeg}.

In the case~$k=3$, the comparison might very well go the other way around. Indeed, this would be the case if our \Cref{conj:exactvalue} about the minimal normalized stretch factors of orientation-reversing fully-punctured pseudo-Anosov maps holds. 

The comparison in \Cref{cor:maincor} agrees with the comparison of the minimal spectral radii of reciprocal and skew-reciprocal primitive matrices~$A\in\mathrm{GL}_{2k}(\ZZ)$ presented in \cite[Theorem~1.1]{Lie23}. Interestingly, by work of Liechti~\cite[Theorem~2 and Proposition~14]{Lie19}, extending this comparison to the class of all  matrices~$A\in\mathrm{GL}_{2k}(\ZZ)$ obtained from the action on the first homology of closed surfaces induced by orientation-preserving and orientation-reversing pseudo-Anosov maps would yield a new proof of a theorem of Dimitrov~\cite{Dim19}, which for a long time was known as the conjecture of Schinzel and Zassenhaus~\cite{SZ65}. In this context, our \Cref{cor:maincor} can be seen as a geometric version of this proof idea. 

Concerning the set of normalized stretch factors of orientation-reversing fully-punctured pseudo-Anosov maps, we have the following result. 

\begin{thm}\label{thm:set}
The smallest three elements of the set of normalized stretch factors of orientation-reversing fully-punctured pseudo-Anosov maps are~$\mu, \sigma$ and~$\mu^2$. Furthermore, the set of normalized stretch factors of orientation-reversing fully-punctured pseudo-Anosov maps contains a dense subset of~$[\sigma^2, \infty)$. 
\end{thm}

\textbf{Outline of the proof of \Cref{thm:mainthm} and  \Cref{thm:set}}

For~\Cref{thm:mainthm}, we follow the approach by Hironaka and Tsang~\cite{HT22}. They develop the theory of standardly embedded train tracks, and study the action on the real edges of the train track induced by the pseudo-Anosov map. In the orientation-preserving case, this allows one to reduce the problem to the study of minimal spectral radii of reciprocal primitive matrices, which are well-known by work of McMullen~\cite[Theorem~1.1]{McM15}, at least in even dimensions. 

Initially, our hope was that the adaptation of McMullen’s result to skew-reciprocal matrices, by Liechti~\cite[Theorem~1.5]{Lie23}, directly allows us to conclude an analogous result for orientation-reversing pseudo-Anosov maps. However, it turns out that in this case the action induced on the real edges need not be skew-reciprocal, but is instead skew-reciprocal up to cyclotomic factors: the eigenvalues are invariant under the transformation~$t\mapsto -t^{-1}$, except for possibly some roots of unity. See~\Cref{rmk:skewreciprocalfail} for an explicit example.

The main technical contribution of this article is to extend certain cases of the analysis of minimal spectral radii of skew-reciprocal primitive matrices, carried out by Liechti~\cite{Lie23}, to the more general class of matrices that are skew-reciprocal up to cyclotomic factors. 

For the proof of~\Cref{thm:set}, we study which elements in the set of normalized stretch factors of orientation-preserving fully-punctured pseudo-Anosov maps, determined by Tsang~\cite{Tsa23}, admit an orientation-reversing square root. To this end, we use a theorem by Strenner and Liechti~\cite[Theorem~1.10]{LS20}, stating that the Galois conjugates of the stretch factor of an orientation-reversing pseudo-Anosov map cannot lie on the unit circle. 

However, our approach is not practicable for analyzing the dense subset of normalized stretch factors of orientation-preserving maps in the interval~$[\mu^4, \sigma^4)$. This is the reason why we are currently unable to describe the normalized stretch factors of orientation-reversing maps in~$(\mu^2,\sigma^2)$. In particular, it remains an open problem to determine whether~$\sigma^2$ is the minimal accumulation point of the set of normalized stretch factors of fully-punctured orientation-reversing pseudo-Anosov maps. 

\textbf{Organization.} 
In~\Cref{sec:background} we introduce the necessary background material on pseudo-Anosov maps and standardly embedded train tracks. In~\Cref{sec:ncsr} we discuss skew-reciprocity up to cyclotomic factors and show that the action of an orientation-reversing pseudo-Anosov map induced on the real edges of a standardly embedded train track is of this kind. The proof of~\Cref{thm:mainthm} occupies the two sections that follow: in~\Cref{section:curvegraphanalysis} we deduce the lower bounds, and in~\Cref{sec:sharpness} we describe examples providing the asymptotic sharpness. In \Cref{sec:thmsetproof} we provide a short proof of \Cref{thm:set}. Finally, in~\Cref{sec:questions}, we discuss some open questions.

\textbf{Acknowledgements.}
The first author would like to thank the University of Fribourg for excellent working conditions. 
This project was completed while the third author is a CRM-ISM postdoctoral fellow based at CIRGET. He would like to thank the center for its support. 
All three authors would like to thank the anonymous referee for their suggestions, which improved the readability of the paper.

\section{Background}
\label{sec:background}

In this section, we recall some background material from \cite{HT22}. The main goal is to explain \Cref{thm:standardlyembeddedtt}, which is essentially a summary of results in \cite[Section 3]{HT22} but allowing for orientation-reversing maps.

\subsection{Pseudo-Anosov maps}

In this paper, a \textbf{finite-type surface} will mean a closed oriented surface with finitely many points, which we call the \textbf{punctures}, removed. 

\begin{defn} \label{defn:pamap}
A homeomorphism $f$ on a finite-type surface $S$ is said to be \textbf{pseudo-Anosov} if there exists a pair of singular measured foliations $(\ell^s,\mu^s)$ and $(\ell^u,\mu^u)$ such that:
\begin{enumerate}
    \item Away from a finite collection of \textbf{singular points}, which includes the punctures, $\ell^s$ and $\ell^u$ are locally conjugate to the foliations of $\mathbb{R}^2$ by vertical and horizontal lines respectively.
    \item Near a singular point, $\ell^s$ and $\ell^u$ are locally conjugate to either
    \begin{itemize}
        \item the pull back of the foliations of $\mathbb{R}^2$ by vertical and horizontal lines by the map $z \mapsto z^{\frac{n}{2}}$ respectively, for some $n \geq 3$, or
        \item the pull back of the foliations of $\mathbb{R}^2 \backslash \{(0,0)\}$ by vertical and horizontal lines by the map $z \mapsto z^{\frac{n}{2}}$ respectively, for some $n \geq 1$.
    \end{itemize}
    \item $f_* (\ell^s,\mu^s) = (\ell^s,\lambda^{-1} \mu^s)$ and $f_* (\ell^u,\mu^u) = (\ell^u,\lambda \mu^u)$ for some $\lambda=\lambda(f)>1$.
\end{enumerate} 
We call $(\ell^s,\mu^s)$ and $(\ell^u,\mu^u)$ the \textbf{stable} and \textbf{unstable} measured foliations respectively. We call $\lambda(f)$ the \textbf{stretch factor} of $f$ and call $\lambda(f)^{|\chi(S)|}$ the \textbf{normalized stretch factor} of $f$.
\end{defn}

\begin{defn} \label{defn:fullypunctured}
$f$ is said to be \textbf{fully-punctured} if the set of singular points equals the set of punctures. In this case, we will denote the set of punctures by $\mathcal{X}$. Note that $f$ acts on $\mathcal{X}$ by some permutation, hence it makes sense to talk about the orbits of punctures under the action of $f$, or \textbf{puncture orbits} for short.
\end{defn}

\subsection{Standardly embedded train tracks}

A \textbf{train track} $\tau$ on a finite-type surface $S$ is an embedded finite graph with a partition of half-edges incident to each vertex into two nonempty subsets, called the \textbf{sides} of the vertex. We refer to the data of the partition as the \textbf{smoothing}, and incorporate this data by choosing some tangent line at each vertex $v$, and arranging the half-edges in the two sides to be tangent to the two sides of the line.

A \textbf{boundary component} of $\tau$ is a boundary component of a complementary region of $\tau$ in $S$. In this paper, it will always be the case that $\tau$ is a deformation retract of $S$, hence the boundary components of $\tau$ are in canonical one-to-one correspondence with the punctures of $S$.
A boundary component is \textbf{$n$-pronged} if it has $n$ cusps, i.e. $n$ nonsmooth points.

\begin{defn} \label{defn:standardlyembedded}
Let $\partial \tau = \partial_I \tau \sqcup \partial_O \tau$ be a partition of the boundary components of $\tau$ into a nonempty set of \textbf{inner boundary components} and a nonempty set of \textbf{outer boundary components} respectively.

A train track $\tau$ is said to be \textbf{standardly embedded (with respect to $(\partial_I \tau, \partial_O \tau)$)} if its set of edges $E(\tau)$ can be partitioned into a set of \textbf{infinitesimal edges} $E_{\infs}(\tau)$ and a set of \textbf{real edges} $E_{\real}(\tau)$, such that:
\begin{itemize}
    \item The smoothing at each vertex is defined by separating the infinitesimal edges and the real edges.
    \item The union of infinitesimal edges is a disjoint union of cycles, which we call the \textbf{infinitesimal polygons}.
    \item The infinitesimal polygons are exactly the inner boundary components of $\tau$.
\end{itemize}
\end{defn}

Let $\tau$ and $\tau'$ be train tracks on $S$. A \textbf{train track map} is a map $f:S \to S$ that sends vertices of $\tau$ to vertices of $\tau'$ and smooth edge paths of $\tau$ to smooth edge paths of $\tau'$. 

Two examples of train track maps are the \textbf{subdivision move} on an edge $e$, and the \textbf{elementary folding move} on a pair of edges $(e_1,e_2)$ that are adjacent on one side of a vertex. We pictorially recall the definitions of these in \Cref{fig:elementarymoves} top and bottom, and refer the reader to \cite[Definitions 3.14 and 3.15]{HT22} for the precise descriptions.

\begin{figure}
    \centering
    \fontsize{8pt}{8pt}\selectfont
\begingroup%
  \makeatletter%
  \providecommand\color[2][]{%
    \errmessage{(Inkscape) Color is used for the text in Inkscape, but the package 'color.sty' is not loaded}%
    \renewcommand\color[2][]{}%
  }%
  \providecommand\transparent[1]{%
    \errmessage{(Inkscape) Transparency is used (non-zero) for the text in Inkscape, but the package 'transparent.sty' is not loaded}%
    \renewcommand\transparent[1]{}%
  }%
  \providecommand\rotatebox[2]{#2}%
  \newcommand*\fsize{\dimexpr\f@size pt\relax}%
  \newcommand*\lineheight[1]{\fontsize{\fsize}{#1\fsize}\selectfont}%
  \ifx\svgwidth\undefined%
    \setlength{\unitlength}{223.32376315bp}%
    \ifx\svgscale\undefined%
      \relax%
    \else%
      \setlength{\unitlength}{\unitlength * \real{\svgscale}}%
    \fi%
  \else%
    \setlength{\unitlength}{\svgwidth}%
  \fi%
  \global\let\svgwidth\undefined%
  \global\let\svgscale\undefined%
  \makeatother%
  \begin{picture}(1,0.5082033)%
    \lineheight{1}%
    \setlength\tabcolsep{0pt}%
    \put(0,0){\includegraphics[width=\unitlength,page=1]{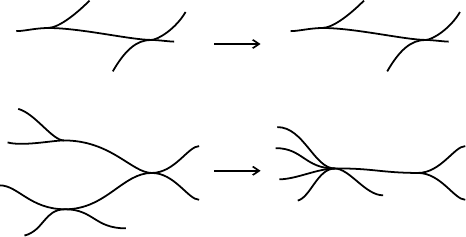}}%
    \put(0.2221499,0.44673983){\color[rgb]{0,0,0}\makebox(0,0)[lt]{\lineheight{1.25}\smash{\begin{tabular}[t]{l}$e$\end{tabular}}}}%
    \put(0.22329358,0.06094524){\color[rgb]{0,0,0}\makebox(0,0)[lt]{\lineheight{1.25}\smash{\begin{tabular}[t]{l}$e_2$\end{tabular}}}}%
    \put(0.22329164,0.1987999){\color[rgb]{0,0,0}\makebox(0,0)[lt]{\lineheight{1.25}\smash{\begin{tabular}[t]{l}$e_1$\end{tabular}}}}%
    \put(0,0){\includegraphics[width=\unitlength,page=2]{elementarymoves.pdf}}%
  \end{picture}%
\endgroup%

    \caption{Top: The subdivision move on $e$. Bottom: The elementary folding move on $(e_1,e_2)$.}
    \label{fig:elementarymoves}
\end{figure}

The \textbf{transition matrix} of a train track map $f$ is the matrix $f_* \in M_{E(\tau') \times E(\tau)}(\mathbb{Z}_{\geq 0})$ whose $(e',e)$-entry is given by the number of times $f(e)$ passes through $e'$.

\begin{thm} \label{thm:standardlyembeddedtt}
Let $f:S \to S$ be a fully-punctured pseudo-Anosov map with at least two puncture orbits. Then $f$ is homotopic to a train track map on a standardly embedded train track $\tau$. This train track map, which we will denote by $f$ as well, has the following properties:
\begin{itemize}
    \item $f$ is a composition $\sigma f_n \cdots f_1$, where each $f_i$ is a subdivision move or an elementary folding move, and $\sigma$ is an isomorphism of train tracks.
    \item The transition matrix of $f$ is of the form 
    $$f_* = \begin{bmatrix}
    P & * \\
    0 & f_*^\real
    \end{bmatrix}$$
    where $P$ is a permutation matrix and $f_*^\real$ is a $|\chi(S)|$-by-$|\chi(S)|$ primitive matrix whose spectral radius equals the stretch factor of $f$.
\end{itemize}
\end{thm}
\begin{proof}
This is essentially shown over \cite[Propositions 3.10, 3.12, 3.13, 3.21]{HT22}. Now, strictly speaking, \cite{HT22} only deals with orientation-preserving pseudo-Anosov maps, but this hypothesis is not used in the proofs of these propositions. In the following, we give a quick summary of these proofs for the reader's convenience and to demonstrate that they carry through in the orientation-reversing setting.

Recall from \cite[Definition 3.4]{HT22} that a \textbf{stable prong} at a puncture $x$ is a connected subset of a stable leaf that limits to $x$. A \textbf{stable star} at $x$ is a maximal disjoint union of prongs at $x$. In particular, a stable star at $x$ has $n$ connected components when $x$ is an $n$-pronged puncture. A \textbf{side} of a stable star is the union of two adjacent prongs.
An \textbf{unstable prong} and an \textbf{unstable star} at $x$ are defined similarly.

Fix a partition $\mathcal{X}=\mathcal{X}_I \sqcup \mathcal{X}_O$ of the set of punctures of $S$ into two nonempty $f$-invariant subsets.
We first construct a partition of $S$ into rectangles:
For each $x \in \mathcal{X}_O$, let $\sigma^u_x$ be the unstable star at $x$ for which each of its prongs has $\mu^s$-length 1. Then for each $x \in \mathcal{X}_I$, we construct a stable star $\sigma^s_x$ at $x$ by extending the stable prongs at $x$ until it bumps into some $\sigma^u_x$. Finally, for each $x \in \mathcal{X}_O$, we extend $\sigma^u_x$ by extending the prongs of $\sigma^u_x$ until it bumps into some $\sigma^s_x$. It is straightforward to check that each complementary region of $\bigcup_{x \in \mathcal{X}_I} \sigma^s_x \cup \bigcup_{x \in \mathcal{X}_O} \sigma^u_x$ is a rectangle. 

From this partition, we can construct the standardly embedded train track $\tau$ by taking its set of vertices to be in one-to-one correspondence with pairs of adjacent prongs of the stable stars $\sigma^s_x$. The infinitesimal edges of $\tau$ are taken to be in one-to-one correspondence with the prongs of $\sigma^s_x$, with their endpoints at the two vertices corresponding to the two pairs the prong lies in. The real edges of $\tau$ are taken to be in one-to-one correspondence with the rectangles of the partition, with their endpoints at the two vertices corresponding to pairs on which the two stable sides of the rectangle lies along.
We illustrate a local picture of this construction in \Cref{fig:partitiontott}.
In this paper, we draw stable leaves in red and unstable leaves in blue.

\begin{figure}
    \centering
    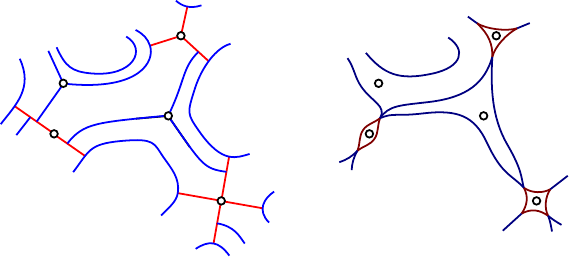
    \caption{A local example of the construction going from the partition by rectangles to the standardly embedded train track.}
    \label{fig:partitiontott}
\end{figure}

By construction, $\tau$ is a deformation retract of $S$, hence we can compute 
$$\chi(S) = \chi(\tau) = |V| - (|E_{\infs}|+|E_{\real}|) = |V| - (|V|+|E_{\real}|) = -|E_{\real}|.$$

Observe that $f^{-1}(\bigcup_{x \in \mathcal{X}_I} \sigma^s_x) \supset \bigcup_{x \in \mathcal{X}_I} \sigma^s_x$ and $f^{-1}(\bigcup_{x \in \mathcal{X}_O} \sigma^u_x) \subset \bigcup_{x \in \mathcal{X}_O} \sigma^u_x$. We can choose an increasing sequence of stable stars interpolating from $\bigcup_{x \in \mathcal{X}_I} \sigma^s_x$ to $f^{-1}(\bigcup_{x \in \mathcal{X}_I} \sigma^s_x)$ and a decreasing sequence of unstable stars interpolating from $\bigcup_{x \in \mathcal{X}_O} \sigma^u_x$ to $f^{-1}(\bigcup_{x \in \mathcal{X}_O} \sigma^u_x)$, so that at each stage the stable and unstable stars partition $S$ into rectangles. By repeating our construction above, we get a sequence of train tracks that differ by subdivision and elementary folding moves. We illustrate an example of a stage of this sequence in \Cref{fig:widerfolding}. Since the last train track in this sequence is constructed from $f^{-1}(\bigcup_{x \in \mathcal{X}_I} \sigma^s_x) \cup f^{-1}(\bigcup_{x \in \mathcal{X}_O} \sigma^u_x)$, $f$ induces a train track isomorphism from this train track back to $\tau$. This shows the first item in the statement.

\begin{figure}
    \centering
    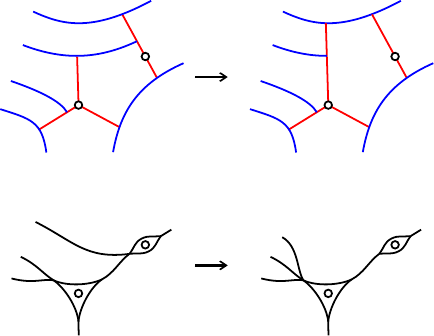
    \caption{As we expand the stable stars and shrink the unstable stars, the train track undergoes subdivision and elementary folding moves.}
    \label{fig:widerfolding}
\end{figure}

By construction, $f$ maps each infinitesimal edge to a single infinitesimal edge, hence if we list the infinitesimal edges in front of the real edges, the upper and lower left blocks of the transition matrix are as claimed in the second item of the statement.

It remains to show the claimed properties of the lower right block $f_*^\real$. 
To show primitiveness, notice that the $(e',e)$-entry of $f_*^k$ is given by the number of times $f^k(e)$ passes through $e'$. Using the fact that each leaf of $\ell^u$ is dense in $S$ (\cite[Proposition 9.6]{FLP79}), one can deduce that for each $e'$ there exists $k_{e'}$ such that the $(e',e)$-entry of $(f_*^\real)^{k_{e'}}$ is positive for each $e$. This implies that $(f_*^\real)^{\prod k_{e'}}$ is positive, hence $f_*^\real$ is primitive.

To find the spectral radius of $f_*$, we define an eigenvector $u$ by taking its $e$-entry to be the $\mu^u$-measure of the corresponding rectangle $R$. By definition, the $e$-entry of $f_*^\real u$ is the $f_*\mu^u$-measure of $R$, which is $\lambda$ times the $e$-entry of $u$. Hence $u$ is a $\lambda$-eigenvector of $f_*^\real$. Since $u$ is positive, by the Perron-Frobenius theorem, $\lambda$ is the spectral radius of $f_*^\real$. 
\end{proof}

For future convenience, we will refer to a standardly embedded train track $\tau$ as in \Cref{thm:standardlyembeddedtt} as a \textbf{$f$-invariant} standardly embedded train track, and we will refer to the block $f^\real_*$ as the \textbf{real transition matrix}.

\section{Skew-reciprocity up to cyclotomic factors}
\label{sec:ncsr}

Recall that a polynomial $p \in \ZZ[t]$ is \textbf{reciprocal} if its roots are invariant under the transformation $t\mapsto t^{-1}$.
A polynomial $q \in \ZZ[t]$ is \textbf{skew-reciprocal} if its roots are invariant under the transformation $t\mapsto -t^{-1}$.
A square matrix is said to be \textbf{reciprocal/skew-reciprocal} if its characteristic polynomial is reciprocal/skew-reciprocal, respectively.

In \cite{HT22}, it was shown that if $f$ is an orientation-preserving pseudo-Anosov map, then the transition matrix $f_*$ preserves a skew-symmetric bilinear form $\omega$, called the Thurston symplectic form. Despite its name, $\omega$ is in general degenerate, but by showing that $f_*$ acts on the radical $\rad(\omega)$ by a cyclotomic, thus reciprocal matrix, one can conclude that $f_*$ is reciprocal nevertheless.

Now if instead $f$ is orientation-reversing, then $f_*$ would send $\omega$ to $-\omega$. One might be tempted to conclude from this that $f_*$ would be a skew-reciprocal matrix. However, this is incorrect since it fails to account for the action of $f_*$ on $\rad(\omega)$. What turns out to be true instead is that $f_*$ satisfies a weaker property as follows:

\begin{defn} \label{defn:noncycloskewreciprocal}
A polynomial $r \in \ZZ[t]$ is \textbf{skew-reciprocal up to cyclotomic factors} if its roots, with possibly the exception of roots of unity, are invariant under the transformation $t\mapsto -t^{-1}$. Equivalently, the polynomial $r$ is the product of any number of cyclotomic factors and a skew-reciprocal polynomial. 

A square matrix is \textbf{skew-reciprocal up to cyclotomic factors} if its characteristic polynomial is skew-reciprocal up to cyclotomic factors.
\end{defn}

In \Cref{subsec:thurstonsympform}, we will explain why $f_*$ and $f^\real_*$ are skew-reciprocal up to cyclotomic factors. In \Cref{subsec:paritycondition}, we will show a simple necessary condition regarding the coefficients of a polynomial that is skew-reciprocal up to cyclotomic factors.

\subsection{The Thurston symplectic form} \label{subsec:thurstonsympform}

Let $\tau$ be a train track. The \textbf{weight space} $\mathcal{W}(\tau)$ of $\tau$ is the subspace of $\mathbb{R}^{E(\tau)}$ consisting of elements $(w_e)$ that satisfy $\sum_{e \in E^1_v} w_e=\sum_{e \in E^2_v} w_e$ at every vertex $v$, where $E^1_v \sqcup E^2_v$ is the smoothing at $v$.

\begin{defn} \label{defn:thurstonform}
Let $w, w' \in \mathcal{W}(\tau)$. We define
$$\omega(w,w') = \sum_{v \in V(\tau)} \sum_{\text{$e_1$ left of $e_2$}} (w_{e_1} w'_{e_2} - w_{e_2} w'_{e_1})$$
where the second summation is taken over all pairs of edges $(e_1,e_2)$ on a side of $v$ for which $e_1$ is on the left of $e_2$. 

Then $\omega$ is clearly a skew-symmetric bilinear form on $\mathcal{W}(\tau)$. We call $\omega$ the \textbf{Thurston symplectic form}.
\end{defn}

\begin{prop} \label{prop:weightspaceses}
Let $f$ be a fully-punctured pseudo-Anosov map with at least two puncture orbits. Let $\tau$ be an $f$-invariant standardly embedded train track. Then there exists a linear map $T$ and a permutation matrix $P \in M_{V(\tau) \times V(\tau)}(\mathbb{R})$ which fit into the commutative diagram
\begin{center}
\begin{tikzcd}
0 \arrow[r] & \mathcal{W}(\tau) \arrow[r] \arrow[d, "f_*"] & \mathbb{R}^{E(\tau)} \arrow[r, "T"] \arrow[d, "f_*"] & \mathbb{R}^{V(\tau)} \arrow[d, "P"] \\
0 \arrow[r] & \mathcal{W}(\tau) \arrow[r] & \mathbb{R}^{E(\tau)} \arrow[r, "T"] & \mathbb{R}^{V(\tau)} 
\end{tikzcd}
\end{center}
Moreover, if $f$ is orientation-reversing, then $f_*:\mathcal{W}(\tau) \to \mathcal{W}(\tau)$ sends $\omega$ to $-\omega$.
\end{prop}
\begin{proof}
The first statement follows from \cite[Proposition 4.2]{HT22}. Again, strictly speaking \cite{HT22} only deals with orientation-preserving maps, but this hypothesis is not used in the proof of this proposition. We give a quick summary of the proof to demonstrate this.

For each vertex $v$, we define a linear map $T_v:\mathbb{R}^{E(\tau)} \to \mathbb{R}$ by $T_v(w) = \sum_{e \in E_{v,\real}} w_e - \sum_{e \in E_{v,\infs}} w_e$.
The map $T$ is then defined by $T(w)=(T_v(w))$. By definition, $\mathcal{W}(\tau) = \ker T$. 

We then define 
$$P_{v',v}=
\begin{cases}
1 \text{ , if $f$ sends $v$ to $v'$} \\
0 \text{ , otherwise.}
\end{cases}$$
It is straightforward to check that $P T = T f_*$.

For the moreover statement, we use the factorization of $f$ as $\sigma f_n \cdots f_1$ from \Cref{thm:standardlyembeddedtt}. By \cite[Lemmas 4.4 and 4.5]{HT22}, each subdivision move and elementary folding move preserves $\omega$. Meanwhile, the train track isomorphism $\sigma$ is orientation-reversing if and only if $f$ is orientation-reversing. In this case, $\sigma$ sends $\omega$ to $-\omega$.
\end{proof}

We can further separate the action of $f_*$ on $\mathcal{W}(\tau)$ into two parts by considering the radical of $\omega$.

Let $c$ be an even-pronged boundary component of $\tau$. Label the sides of $\partial c$ by $I_1,...,I_n$ in a cyclic order. The \textbf{radical element} is the element $r_c$ of $\mathcal{W}(\tau)$ where we assign to each edge on $I_k$ a weight of $(-1)^k$.

\begin{prop} \label{prop:radicalreciprocal}
Let $f$ be a fully-punctured pseudo-Anosov map with at least two puncture orbits. Let $\tau$ be an $f$-invariant standardly embedded train track. Let $\partial_{\mathrm{even}} \tau$ be the set of even-pronged punctures. Then there exists a signed permutation matrix $P$ which fits into the commutative diagram
\begin{center}
\begin{tikzcd}
\mathbb{R}^{\partial_{\mathrm{even}} \tau} \arrow[r] \arrow[d, "P"] & \mathcal{W}(\tau) \arrow[r] \arrow[d, "f_*"] & \mathcal{W}(\tau)/\rad(\omega) \arrow[r] \arrow[d, "f_*"] & 0 \\
\mathbb{R}^{\partial_{\mathrm{even}} \tau} \arrow[r] & \mathcal{W}(\tau) \arrow[r] & \mathcal{W}(\tau)/\rad(\omega) \arrow[r] & 0 
\end{tikzcd}
\end{center}
Moreover, if $f$ is orientation-reversing, then $f_*:\mathcal{W}(\tau)/\rad(\omega) \to \mathcal{W}(\tau)/\rad(\omega)$ is skew-reciprocal.
\end{prop}
\begin{proof}
\cite[Proposition 5.5]{HT22} states that $\rad(\omega) = \spn \{r_c\}$. Meanwhile, observe that $f_*(r_c) = \pm r_{f(c)}$. This implies that $f_*$ acts on $\mathbb{R}^{\partial_{\mathrm{even}} \tau}$ via a signed permutation matrix $P$. 

For the moreover statement, $\omega$ descends to a symplectic form on $\mathcal{W}(\tau)/\rad(\omega)$. $f_*:\mathcal{W}(\tau)/\rad(\omega) \to \mathcal{W}(\tau)/\rad(\omega)$ is anti-symplectic with respect to this symplectic form, hence by \cite[Proposition 4.2]{LS20}, this matrix is skew-reciprocal.
\end{proof}

\begin{prop} \label{prop:orirevpamaptononcycloskewreciprocal}
Let $f$ be an orientation-reversing fully-punctured pseudo-Anosov map with at least two puncture orbits. Let $\tau$ be an $f$-invariant standardly embedded train track. Then the real transition matrix $f^\real_*$ is skew-reciprocal up to cyclotomic factors.
\end{prop}
\begin{proof}
The only roots of the characteristic polynomial of a signed permutation matrix are roots of unity. From \Cref{prop:radicalreciprocal}, we deduce that $f_*:\mathcal{W}(\tau) \to \mathcal{W}(\tau)$ is skew-reciprocal up to cyclotomic factors. From \Cref{prop:weightspaceses}, it follows that $f_*: \mathbb{R}^{E(\tau)} \to \mathbb{R}^{E(\tau)}$ is skew-reciprocal up to cyclotomic factors. Finally, from the second item in \Cref{thm:standardlyembeddedtt}, it follows that $f^\real_*$ is skew-reciprocal up to cyclotomic factors.
\end{proof}

\begin{rmk} \label{rmk:skewreciprocalfail}
We exhibit an example that shows that \Cref{prop:orirevpamaptononcycloskewreciprocal} becomes false if we omit `up to cyclotomic factors'.

Let $S$ be the surface $(\mathbb{R}^2 \backslash (\frac{1}{2}\mathbb{Z})^2)/\mathbb{Z}^2$, which is a torus with four punctures. The matrix $\begin{bmatrix} 1 & 1 \\ 1 & 0 \end{bmatrix}$ induces an orientation-reversing pseudo-Anosov map $f$ on $S$. It is straightforward to check that $f$ has two puncture orbits. 

The standardly embedded train track $\tau$ illustrated in \Cref{fig:skewreciprocalfail} top is a $f$-invariant train track. In the basis $(a,b,c,d)$ as indicated in \Cref{fig:skewreciprocalfail} bottom left, one computes
$$f^\real_* = 
\begin{bmatrix}
0 & 0 & 1 & 1 \\
1 & 0 & 0 & 0 \\
1 & 1 & 0 & 0 \\
0 & 0 & 1 & 0 
\end{bmatrix}.$$
The characteristic polynomial of $f^\real_*$ is $t^4-t^2-2t-1=(t^2-t-1)(t^2+t+1)$. Hence in this example, $f^\real_*$ is skew-reciprocal up to cyclotomic factors but not skew-reciprocal.

\begin{figure}
    \centering
    \fontsize{8pt}{8pt}\selectfont
    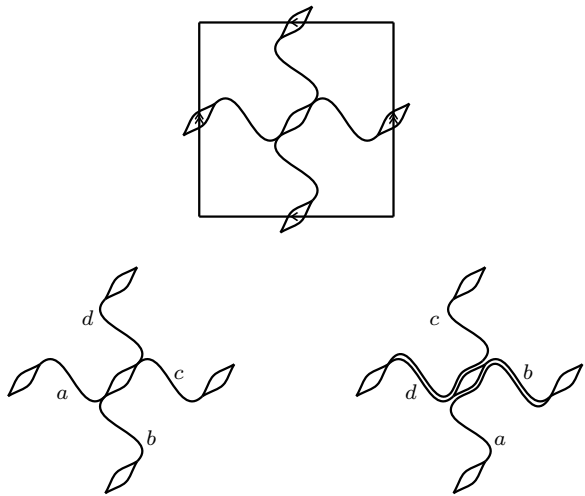
    \caption{An example that shows that \Cref{prop:orirevpamaptononcycloskewreciprocal} becomes false if we omit `up to cyclotomic factors'}
    \label{fig:skewreciprocalfail}
\end{figure}

\end{rmk}

\subsection{A parity condition} \label{subsec:paritycondition}

Recall that a polynomial $p(t) = a_nt^n+\cdots+a_1t+a_0\in\ZZ[t]$ is reciprocal if and only if $p(t) = \pm t^n p(t^{-1})$. This polynomial equation transforms to the equation~$a_d = \pm a_{n-d}$ for all pairs of coefficients~$a_d$ and~$a_{n-d}$. 

Similarly, a polynomial $q(t) = b_mt^m+\cdots+b_1t+b_0\in\ZZ[t]$ is skew-reciprocal if and only if $q(t) = \pm t^{m}q(-t^{-1})$, which in turn implies the equation~$b_d = \pm (-1)^{m-d}b_{m-d}$ on the level of coefficients.

In comparison, the coefficients of a polynomial that is skew-symmetric up to cyclotomic factors may not be symmetric up to signs. For example, consider the polynomial $(t^2-t-1)(t^2+t+1)=t^4-t^2-2t-1$ as in \Cref{rmk:skewreciprocalfail}. 

Nevertheless, we have the following lemma, which states a symmetry condition that holds for any product of a reciprocal and a skew-reciprocal polynomial. 
    
\begin{lemma}
    \label{lemma:parity}
    Suppose~$r(t) = p(t)q(t) = c_kt^k + \cdots + c_1t+c_0$, where~$p \in \ZZ[t]$ is a reciprocal polynomial and~$q \in \ZZ[t]$ is a skew-reciprocal polynomial. Then~$c_d + c_{k-d}$ is even for all~$d=0,\dots,k.$
\end{lemma}

\begin{proof}
Let $a_i$, $b_j$ denote the coefficients of $p$ and $q$ respectively.
Considering the coefficients in~$\ZZ/2\ZZ$, we have~$a_d \equiv a_{n-d}~(\mathrm{mod}~2)$ and~$b_d \equiv b_{m-d}~(\mathrm{mod}~2)$. Hence the coefficients of~$r=pq$ also satisfy~$c_d\equiv c_{k-d}~(\mathrm{mod}~2)$, which means exactly that ~$c_d + c_{k-d}$ is even.
\end{proof}

Clearly, \Cref{lemma:parity} is only a rough necessary condition for polynomials that are skew-reciprocal up to cyclotomic factors. In \Cref{section:curvegraphanalysis} we will need to refine out analysis in order to obtain an asymptotic classification result for small normalized spectral radii of matrices~$A \in \mathrm{GL}_n(\ZZ)$ that are primitive and skew-reciprocal up to cyclotomic factors.

\section{Analysis of small curve graphs}
\label{section:curvegraphanalysis}

The objective of this section is to prove the following theorem.

\begin{thm} \label{thm:skucspectralradius}
Let $A$ be a primitive matrix in $\mathrm{GL}_n(\ZZ)$, $n \geq 4$, that is skew-reciprocal up to cyclotomic factors. Then $\rho(A)^n \geq 3+2\sqrt{2}$.
\end{thm}

Together with \Cref{thm:standardlyembeddedtt} and \Cref{prop:orirevpamaptononcycloskewreciprocal}, this implies the inequality in \Cref{thm:mainthm}.

The starting point of the proof is the following proposition.

\begin{prop}
    \label{prop:smallcurvegraphs}
    Let~$A \in \mathrm{GL}_n(\ZZ)$ be a primitive matrix with spectral radius~$\rho(A) > 1$ such that~$\rho(A)^n < 3+2\sqrt{2}$. Then the reciprocal of the characteristic polynomial of~$A$ has one of the following forms:
    \begin{align}
         Q(t) &= 1-t^a -t^b \label{2A_1}\tag{$2A_1$} \\
         Q(t) &= 1- t^a -t^b -t^{c} \label{3A_1}\tag{$3A_1$}\\
         Q(t) &= 1-t^a -t^b -t^c -t^d \label{4A_1}\tag{$4A_1$}\\
         Q(t) &= 1-t^a-t^b-t^c-t^d-t^e \label{5A_1}\tag{$5A_1$}\\
         Q(t) &= 1-t^a-t^b-t^c+t^{a+b} \label{A^*_2}\tag{$A^*_2$}
    \end{align}
    for suitable positive parameters~$a,b,c,d,e \in\ZZ$.
\end{prop}

In the statement of \Cref{prop:smallcurvegraphs}, each polynomial is labelled by the corresponding curve graph, as we will explain in the proof below.

\begin{proof}[Proof of Proposition~\ref{prop:smallcurvegraphs}]
This result is contained in \cite{McM15}. We briefly indicate how to deduce this exact statement from the technology in that paper.

To a matrix~$A \in \mathrm{GL}_n(\ZZ)$ we associate a directed graph~$\Gamma$ that has~$n$ vertices and directed edges between vertices according to the coefficients of the matrix~$A$. The associated curve graph~$G_A$ has one vertex for every simple closed curve in~$\Gamma$, where a simple closed curve is a directed closed loop that visits every vertex at most once. Further, two vertices of~$G_A$ are connected by an edge if and only if the corresponding simple closed curves have no vertex of~$\Gamma$ in common. Every vertex~$v$ of~$G_A$ has a weight~$w(v)$ describing the number of edges in the associated simple closed curve. 

It is a well-known result from graph theory that the clique polynomial~$Q(t)$ of the weighted graph~$G_A$ is the reciprocal of the characteristic polynomial of the matrix~$A: Q(t) = t^n \chi_A (t^{-1})$. 
In particular the spectral radius $\rho(A)$ equals the reciprocal of the smallest positive root of $Q(t)$, which in turn equals the growth rate of $G_A$ (\cite[Theorem 4.2]{McM15}).

It suffices to describe the possible clique polynomials that can arise for the matrix~$A$. McMullen’s classification of curve graphs with small growth rate implies that the only curve graphs that can arise for the matrix~$A$ are the graphs~$nA_1$ for~$n=2,\dots,5$ or~$A^\ast_2$. To be precise, \cite[Theorem 1.6]{McM15} lists all possible curve graphs that can arise for~$\rho(A)^n < 8$ and \cite[Figure 1]{McM15} lists the minimum growth rates for each of these. We rule out those graphs where we cannot have~$\rho(A)^n < 3+2\sqrt{2}$. The possible clique polynomials for the remaining curve graphs are given in the statement of the proposition. 
\end{proof}

\textbf{Proof of \Cref{thm:skucspectralradius}.}
\Cref{prop:smallcurvegraphs} gives us a list of five different forms of polynomials that we need to consider. We will analyze these one-by-one below. Throughout the analysis we write $\mu = \frac{1+\sqrt{5}}{2}$ for the golden ratio and $\sigma = 1+\sqrt{2}$ for the silver ratio. 
We write $P(t)$ for the characteristic polynomial of $A$ and $Q(t)$ for the reciprocal of $P(t)$, as in \Cref{prop:smallcurvegraphs}.

\subsection*{Case $2A_1$.} In this case~$Q(t) = 1-t^a - t^b$. 
Without loss of generality, we may assume~$b=n$. 
This yields characteristic polynomials of the form~$P(t)=t^{n}-t^a-1.$ 
Observe that $n$ must be even, since otherwise by skew-reciprocity up to cyclotomic factors, $P(t)$ must have an odd number of roots (with multiplicity) on the unit circle, thus $P(t)$ must have $1$ or $-1$ as a root, yet $P(\pm 1)$ is an odd integer so this is impossible. We write $n=2g$, $g \geq 2$.
The only way for $P(t)$ to satisfy the parity condition of \Cref{lemma:parity} is if $a=g$. But then the polynomial is a polynomial in~$t^g$ and $A$ would not be primitive.

\begin{rmk} \label{rmk:2A_1lowdeg}
Here we used the assumption that $n \geq 4$ to have $g \geq 2$, thus obtain a contradiction to $A$ being primitive. When $n=2$, we can have $P(t)=t^2-t-1$, whose normalized largest root is $\mu^2 < 3+2\sqrt{2}$. In fact, this characteristic polynomial is realized by the orientation-reversing fully-punctured pseudo-Anosov map obtained as the action of $\begin{pmatrix} 1 & 1 \\ 1 & 0 \end{pmatrix}$ on the 4-punctured sphere, which has two puncture orbits.
\end{rmk}

\subsection*{Case $3A_1$.} In this case~$Q(t)=1-t^a-t^b-t^c,$ and we may assume~$c=n$. 
This yields characteristic polynomials of the form~$P(t)=t^{n}-t^a-t^b-1.$ 
There are two ways to satisfy the parity condition of \Cref{lemma:parity}: either~$a+b=n$ or~$a=b$. 
Hence the characteristic polynomial has one of the following two forms:
\begin{enumerate}
    \item $P(t) = t^{n}-t^{\frac{n}{2}+d}-t^{\frac{n}{2}-d}-1$, where $0 \leq d < \frac{n}{2}$, or
    \item $P(t) = t^{n}-2t^a-1$, where $0 \leq a < n$.
\end{enumerate}

\begin{claim} \label{claim:3A1claim1}
In case (1), the largest root $\lambda$ of $P(t)$ is strictly increasing in~$d$.
\end{claim}

\begin{proof}[Proof of \Cref{claim:3A1claim1}]
We introduce the function $h(x,s)=x^2-x^{1+s}-x^{1-s}-1$ defined on the domain 
$C = \{ (x,s) \in \RR^2 :\ x>1,\ 0 \leq s<1\}$. Observe $x=\lambda^\frac{n}{2}$ is the largest real root of $h(x,s)$ with $s=\frac{2d}{n}<1$. For any fixed $s$, notice that
$$
\partial^2_x h(x,s) = 2-(1+s)sx^{s-1}+(1-s)sx^{-s-1} > 0.
$$
Thus $\partial_x h(x,s)$ is strictly increasing in $x$. Furthermore $\lim_{x\to \infty} \partial_x h(x,s)=+\infty$ and $h(1,s)=-2<0$, hence the equation $h(x,s)=0$ has a unique real solution $x(s)>1$ depending continuously on $s$. 
Meanwhile, $\partial_s h(x,s) \leq 0$ on $C$, where equality holds if and only if $s=0$. Since $\partial_x h(x(s),s)>0$ for any $s$, we have $x'(s) \geq 0$, equality holds if and only if $s=0$. Thus $x(s)$ is strictly increasing in $s$.
\end{proof}

In particular, the minimal largest root $\lambda$ is attained at~$d=0,$ when $P(t) = t^{n}-2t^{\frac{n}{2}}-1$ has normalised largest root~$3+2\sqrt{2}.$

In case (2), if $P(t)$ is skew-reciprocal, then we must have $n$ even and $a = \frac{n}{2}$, which would reduce to case (1). Hence we can assume that $P(t)$ is skew-reciprocal up to cyclotomic factors but not skew-reciprocal. Then there must exist a root of unity~$\xi$ that is a zero: $\xi^n -2\xi^a = 1.$ The only way for this to happen is if~$\xi^n = \xi^a = -1$. Since $A$ is primitive, $n$ and $a$ are coprime, hence $\xi=-1$ and $n$ and $a$ are odd. Writing $P(t) = t^a (t^{n-a}-1) - (t^a+1)$, we compute that the quotient of $P(t)$ by $t+1$ to be 
$$R(t) = t^a(t^{n-a-1}-t^{n-a-2} + \cdots -1) - (t^{a-1}-t^{a-2} + \cdots +1)$$
Observe that $R(t)$ has no roots of unity as a root. Indeed, by the reasoning above, the only possible root is $\xi =-1$, yet $R(-1) = (n-a) - a = n-2a$ is an odd integer. 
Meanwhile, by inspecting the coefficients of $R(t)$,
$$R(t) = \begin{cases} t^{n-1}-t^{n-2}+ \cdots +t-1 & \text{if $a \geq 3$} \\ t^{n-1}-t^{n-2}+t^{n-3}- \cdots +t^2-t-1 & \text{if $a=1$} \end{cases}.$$
we see that $R(t)$ is not skew-reciprocal, giving us a contradiction.

\begin{rmk} \label{rmk:3A_1lowdeg}
In the analysis above, we used the assumption that $n \geq 4$ to deduce that $R(t)$ is not skew-reciprocal. When $n=3$ and $a=1$, $R(t)=t^2-t-1$ is actually skew-reciprocal. Correspondingly, the characteristic polynomial $P(t)=t^3-2t-1$ is skew-reciprocal up to cyclotomic factors, and has normalized largest root $\mu^3 < 3+2\sqrt{2}$. It is not clear to us at the moment whether this characteristic polynomial is attained by any orientation-reversing fully-punctured pseudo-Anosov map with at least two puncture orbits.
\end{rmk}

\subsection*{Case $4A_1$.} In this case~$Q(t)=1-t^a-t^b-t^c-t^d,$ and we may assume~$d=n$. This yields characteristic polynomials of the form~$P(t) = t^{n}-t^a-t^b-t^c-1$. 
As in case $2A_1$, $n$ must be even. Let us write $n=2g$.

There are three coefficients other than the ones associated to $t^{2g}$ and $-1$.
The parity condition of \Cref{lemma:parity} implies that one of them must be the middle coefficient: say~$b=g$. There are two ways to satisfy the parity condition of \Cref{lemma:parity}: either~$a+c=2g$ or~$a=c$. Hence the characteristic polynomial has one of the following two forms:
\begin{enumerate}
\item $P(t)=t^{2g}-t^{g+d}-t^g-t^{g-d}-1$, where $0 \leq d < g$, or
\item $P(t) = t^{2g}-2t^{a}-t^g-1$, where $0 < a < 2g$.
\end{enumerate}

\begin{claim} \label{claim:4A1claim2}
In case (1), the largest root $\lambda$ of $P(t)$ is strictly increasing in $d$.
\end{claim}

\begin{proof}[Proof of \Cref{claim:4A1claim2}]
We introduce the function $h(x,s)=x^2-x^{1+s}-x-x^{1-s}-1$ defined on the domain 
$C = \{ (x,s) \in \RR^2 :\ x>1,\ 0 \leq s<1\}$. Observe $x=\lambda^g$ is the largest real root of $h(x,s)$ with $s=\frac{d}{g}<1$. For any fixed $s$, notice that
$$
\partial^2_x h(x,s) = 2-(1+s)sx^{s-1}+(1-s)sx^{-s-1} > 0.
$$
Thus $\partial_x h(x,s)$ is strictly increasing in $x$. Furthermore $\lim_{x\to \infty} \partial_x h(x,s)=+\infty$ and $h(1,s)=-3<0$, hence the equation $h(x,s)=0$ has a unique real solution $x(s)>1$ depending continuously on $s$.
Meanwhile, $\partial_s h(x,s) \leq 0$ on $C$, where equality holds if and only if $s=0$. Since $\partial_x h(x(s),s)>0$ for any $s$, we have $x'(s) \geq 0$, equality holds if and only if $s=0$. Thus $x(s)$ is strictly increasing in $s$.
\end{proof}

In particular, the minimal largest root is attained at~$d=0$, when~$P(t)=t^{2g}-3t^g-1$ has normalized largest root~$\left( \cfrac{3+\sqrt{13}}{2}\right)^2 > 3+2\sqrt{2}$.

In the case (2), if $P(t)$ is skew-reciprocal, then we must have $a = g$, which belongs to the former case. Hence we can assume that $P(t)$ is skew-reciprocal up to cyclotomic factors but not skew-reciprocal. Then there must exist a root of unity~$\xi$ that is a zero: $\xi^{2g}-\xi^g-1=2\xi^a$.
We claim that the only way this can happen is if $\xi$ is a sixth root of unity. 

To see this, we write $\xi^g=e^{i\theta}$ and take the norm of both sides of the equation
\begin{align*}
4 &= (\cos 2\theta -\cos \theta - 1)^2 + (\sin 2\theta - \sin \theta)^2 \\
&= (\cos^2 2\theta + \cos^2 \theta +  1 + 2 \cos \theta - 2 \cos 2\theta - 2 \cos 2\theta \cos \theta) \\
&~+ ( \sin^2 2\theta +  \sin^2 \theta - 2 \sin 2\theta \sin \theta) \\
&= 3 + 2 \cos \theta - 2 \cos 2\theta - 2  \cos (2\theta-\theta)  \\
&= 3  - 2 \cos 2\theta. 
\end{align*}
Thus $\cos 2\theta = -\frac{1}{2}$ and $\theta \in \{ \frac{\pi}{3},\frac{2\pi}{3},\frac{4\pi}{3},\frac{5\pi}{3}\}$. For each of the four possibilities for $\theta$, we have:
$$
\begin{array}{c|c|c|c}
\theta & \xi^g & \xi^{2g} & \xi^a = \frac{\xi^{2g}-\xi^g-1}{2} \\
\hline
\frac{2\pi}{3} & e^{\frac{2i\pi}{3}} & e^{\frac{4i\pi}{3}}  & e^{\frac{4i\pi}{3}} \\
\frac{4\pi}{3} & e^{\frac{4i\pi}{3}} & e^{\frac{2i\pi}{3}} & e^{\frac{2i\pi}{3}} \\
\frac{\pi}{3} & e^{\frac{i\pi}{3}} & e^{\frac{2i\pi}{3}} & -1 \\
\frac{5\pi}{3} & e^{\frac{5i\pi}{3}} & e^{\frac{4i\pi}{3}} & -1 
\end{array}
$$
Suppose $\xi$ is a primitive $q^\text{th}$ root of unity. 
For the first two cases, since $\xi^{3g} = \xi^{3a} = 1$, we have $q | \gcd(3g,3a) = 3$, i.e. $q=3$. In turn, we have $a=3A-k$ and $g=3G+k$ for $k=1$ or $2$.
Similarly for the last two cases, since $\xi^{6g} = \xi^{2a} = 1$, we have $q | \gcd(6g,2a) \leq 6$. But $\xi^g$ is a primitive sixth root of unity, so in fact $q=6$. In turn, we have $a=6A+3$ and $g=6G+k$ for $k=\pm 1$.

In particular, this reasoning shows that the first two and the last two cases are mutually disjoint. Correspondingly, $P(t)$ is in one of two possible forms:
\begin{enumerate}[label=(\roman*)]
\item $P(t) = t^{6G+2k}-2t^{3A-k}-t^{3G+k}-1$, where $k=1$ or $2$, and $1 \leq A < 2G+k$, or
\item $P(t) = t^{12G+2k}-2t^{6A+3}-t^{6G+k}-1$, where $k = \pm 1$, and $0 \leq A < 2G$.
\end{enumerate}

\begin{claim} \label{claim:4A1claim1}
In case (i), if $P(t)$ is skew-reciprocal up to cyclotomic factors but not skew-reciprocal, then $k=2$, $G=0$, and $A=1$, i.e. $P(t) = t^4-t^2-2t-1=(t^2-t-1)(t^2+t+1)$ (compare with \Cref{rmk:skewreciprocalfail}) which has normalized largest root $\mu^4 > 3+2\sqrt{2}$.
\end{claim}

\begin{proof}[Proof of \Cref{claim:4A1claim1}]
Assume $P(t)$ is skew-reciprocal up to cyclotomic factors, but not skew-reciprocal. We reasoned above that it has a cubic root of unity as a root. We will compute the quotient 
$R(t)$ of $P(t)$ by $t^2+t+1$.
$$
P(t) = 2(t^3-1)\frac{(t^{6G+2k}-t^{3A-k})}{t^3-1}-(t^{6G+2k}+t^{3G+k}+1).
$$
For the first term, we get
$$
2(t^3-1)\frac{(t^{6G+2k}-t^{3A-k})}{t^3-1} = 2t^{3A-k}(t^3-1)(t^{6G-3A+3k-3}+ \cdots + t^3 + 1).
$$
On the other hand, for the second term, we get
\begin{align*}
t^{6G+2k}+t^{3G+k}+1 &= (t^{6G+2k}-t^{3G+2k})-(t^{3G}-1) + (t^{3G+2k} + t^{3G+k} + t^{3G} ) \\
&= t^{3G+2k}(t^{3G}-1)-(t^{3G}-1)+ t^{3G}(t^{2k}+t^k+1) \\
&= (t^{3G}-1)(t^{3G+2k}-1) + t^{3G}(t^{2k}+t^k+1) \\
&= (t^3-1)(t^{3G-3}+\cdots+t^3+1)(t^{3G+2k}-1) + t^{3G}(t^{2k}+t^k+1).
\end{align*}

Hence
\begin{align*}
R(t) &= 2t^{3A-k}(t-1)(t^{6G-3A+3k-3}+ \cdots + t^3 + 1) \\
&~- (t-1)(t^{3G-3}+\cdots+t^3+1)(t^{3G+2k}-1) - t^{3G}\frac{t^{2k}+t^k+1}{t^2+t+1}.
\end{align*}

We now check that $R(t)$ has no roots of unity as a root. Indeed, by the reasoning above, if $\xi$ is a root then it is a third root of unity.
However, we compute
\begin{align*}
R(\xi) &= \begin{cases} 2(1-\xi^2)(2G-A+1) - 3G-1 & \text{if $k=1$} \\ 2(\xi^2-\xi)(2G-A+2) + 3\xi G+2\xi & \text{if $k=2$} \end{cases} \\
&\neq 0.
\end{align*}
Thus we only need to show that $R(t)$ is not skew-reciprocal, except for the specified case.

For $k=1$, 
\begin{align*}
R(t) &= 2\sum_{i=A}^{2G} (t^{3i} - t^{3i-1})- t^{3G} - \sum_{i=0}^{G-1} (t^{3i+1} - t^{3i})(t^{3G+2}-1) \\
&= t^{6G}-t^{6G-1} +\cdots +t-1
\end{align*}
is not skew-reciprocal.

For $k=2$ and $G > 0$,
\begin{align*}
R(t) &= 2\sum_{i=A}^{2G+1} (t^{3i-1} - t^{3i-2})- (t^2-t+1)t^{3G} - \sum_{i=0}^{G-1} (t^{3i+1} - t^{3i})(t^{3G+4}-1) \\
&= \begin{cases} t^{6G+2}-t^{6G+1} +\cdots +t-1 & \text{if $A >1$} \\ t^{6G+2}-t^{6G+1}+0t^{6G} + \cdots +2t^2-t-1 & \text{if $A=1$} \end{cases}
\end{align*}
is not skew-reciprocal.
If $G=0$, then $A=1$ and $R(t) = t^2-t-1$ is actually skew-reciprocal.
\end{proof}

\begin{claim} \label{claim:4A1claim3}
In case (ii), $P(t)$ is never skew-reciprocal up to cyclotomic factors.
\end{claim}

\begin{proof}[Proof of \Cref{claim:4A1claim3}]
We follow the same strategy as in case (i). We first compute the quotient of $P(t)$ by $t^2-t+1$.
$$
P(t) = t^{12G+2k}-2t^{6A+3}-t^{6G+k}-1 = (t^{12G+2k}-t^{6G+k}+1) -2(t^{6A+3} + 1).
$$
For the first term:
\begin{align*}
t^{12G+2k}-t^{6G+k}+1 &= (t^{12G+2k}-t^{6G+2k})-(t^{6G}-1)+(t^{6G+2k}-t^{6G+k}+t^{6G}) \\
&= (t^{6G}-1)(t^{6G+2k}-1) +t^{6G+k-1}(t^{1+k}-t+t^{1-k}) \\
&= (t^{6G-3}-\cdots +t^3-1)(t^3+1)(t^{6G+2k}-1) +t^{6G+k-1}(t^2-t+1).
\end{align*}

For the second term:
$$t^{6A+3} + 1 = (t^{6A}- \cdots -t^3+1)(t^3+1).$$

Hence
\begin{align*}
R(t) &= (t^{6G-3}-\cdots +t^3-1)(t+1)(t^{6G+2k}-1) +t^{6G+k-1} - 2(t^{6A}- \cdots -t^3+1)(t+1).
\end{align*}

We then check that $R(t)$ has no root of unity as a root. The only possible root $\xi$ satisfies $\xi^3=-1$, and
\begin{align*}
R(\xi) &= \begin{cases} 6\xi G -\xi -2(2A+1)(\xi+1) & \text{if $k=-1$} \\ 6G+1-2(2A+1)(\xi+1) & \text{if $k=1$} \end{cases} \\
&\neq 0
\end{align*}

Finally we observe that
\begin{align*}
R(t) &= \sum_{i=0}^{2G-1} (-1)^{i+1}(t^{3i+1}+t^{3i})(t^{6G+2k}-1) + t^{6G+k-1} - 2\sum_{i=0}^{2A} (-1)^i(t^{3i+1}+t^{3i}) \\
&= \begin{cases} t^{12G+2k-2}+t^{12G+2k-3} - \cdots -t-1 & \text{if $k=1$ or $A < 2G-1$} \\ t^{12G-4}-t^{12G-5}-2t^{12G-6} - \cdots +0t^2-t-1 & \text{if $k=-1$ and $A = 2G-1$} \end{cases}
\end{align*}
is not skew-reciprocal.
\end{proof}

\subsection*{Case $5A_1$.} In this case,~$Q(t)=1-t^a-t^b-t^c-t^d-t^e,$ and we may assume~$e=n$. This yields characteristic polynomials of the form~$t^{n}-t^a-t^b-t^c-t^d-1$. To satisfy the parity condition of \Cref{lemma:parity}, the characteristic polynomial has one of the three possible forms
\begin{enumerate}
\item $P(t) = t^{n}-t^{\frac{n}{2}+a}-t^{\frac{n}{2}+b}-t^{\frac{n}{2}-b}-t^{\frac{n}{2}-a}-1$, for $0\leq a,b<\frac{n}{2}$, or
\item $P(t) = t^{n}-2t^{a}-t^{\frac{n}{2}+b}-t^{\frac{n}{2}-b}-1$, for $0 < a < n$, $0\leq b<\frac{n}{2}$, or
\item $P(t) = t^{n}-2t^{a}-2t^{b}-1$, for $0< a,b<n$.
\end{enumerate}

\begin{claim} \label{claim:5A1claim1}
In case (1), the largest root $\lambda$ of $P(t)$ is strictly increasing in $a$ and $b$.
\end{claim}

\begin{proof}[Proof of \Cref{claim:5A1claim1}]
We introduce the function $h(x,s,u)=x^2-x^{1+s}-x^{1+u}-x^{1-s}-x^{1-u}-1$ defined on the domain 
$C = \{ (x,s,u) \in \RR^3 :\ x>1,\ 0 \leq  s,u <1\}$. Observe that $x=\lambda^{\frac{n}{2}}$ is the largest real root of $h(x,s,u)$ with $s=\frac{2a}{n}<1$ and $u=\frac{2b}{n}<1$. For any fixed $(s,u)$ and $x>1$,
$$
\partial^3_x h(x,s,u) = (1+s)s(1-s) (x^{s-2} - x^{-s-2}) + (1+u)u(1-u) (x^{u-2} - x^{-u-2}) > 0.
$$
Hence $\partial^2_x h(x,s,u)$ is strictly increasing in $x$. Moreover $\partial_x h(1,s,u) = 2-(1+s)-(1+u)-(1-s)-(1-u) =-2 < 0$ and $\partial_x h(x,s,u) > 0$ for large enough $x>1$. Thus $\partial_x h(\cdot,s,u)$ has a unique zero $x_0>1$, and $h(\cdot,s,u)$ is decreasing on $(1,x_0)$ and increasing on $(x_0,\infty)$. Since $h(1,s,u)=-4<0$ and $\lim_{x\to \infty} h(x,s,u)=+\infty$ this proves that the equation $h(x,s,u)=0$ has a unique real solution $x(s,u)> x_0>1$, depending continuously on $(s,u)$. Furthermore, $\partial_s h(x,s,u) \leq 0$ and $\partial_u h(x,s,u) \leq 0$, with equality if and only if $s=0$ and $u=0$ respectively, hence by the implicit function theorem $x(s,u)$ is strictly increasing in $s$ and $u$.
\end{proof}

In particular, the minimal largest root of $P(t)$ is attained at~$a=b=0$, when $P(t)=t^{n}-4t^{\frac{n}{2}}-1$ has normalized largest root~$(2+\sqrt{5})^2 > 3+2\sqrt{2}$. 

In the second case, we claim that the largest root $\lambda$ of $P(t)$ is strictly increasing in $a$ and $b$.

\begin{claim} \label{claim:5A1claim2}
In case (2), the largest root $\lambda$ of $P(t)$ is strictly increasing in $a$ and $b$.
\end{claim}

\begin{proof}[Proof of \Cref{claim:5A1claim2}]
We introduce the function $h(x,s,u)=x^2-2x^s-x^{1+u}-x^{1-u}-1$ defined on the domain 
$C = \{ (x,s,u) \in \RR^3 :\ x>1, 0<s<2 ,0 \leq u <1\}$. Observe that $x=\lambda^{\frac{n}{2}}$ is the largest real root of $h(x,s,u)$ with $s=\frac{2a}{n}$ and $u=\frac{2b}{n}$. We further separate $C$ into $C_1 = \{ (x,s,u) \in C : 0<s \leq 1\}$ and $C_2 = \{ (x,s,u) \in C : 1 \leq s < 2\}$. 

For any fixed $(s,u) \in C_1$,
$$
\partial^2_x h(x,s,u) = 2+2s(1-s)x^{s-2} - (1+u)u x^{u-1} + (1-u)u x^{-u-1} > 0.
$$
Thus $\partial_x h(x,s,u)$ is strictly increasing in $x$. Furthermore $\lim_{x\to \infty} \partial_x h(x,s,u)=+\infty$ and $h(1,s,u)=-4<0$, hence the equation $h(x,s,u)=0$ has a unique real solution $x(s,u)>1$ depending continuously on $s,u$. Furthermore, $\partial_s h(x,s,u) < 0$ and $\partial_u h(x,s,u) \leq 0$, with equality if and only if $u=0$, hence by the implicit function theorem $x(s,u)$ is strictly increasing in $s$ and $u$ in $C_1$.

For any fixed $(s,u) \in C_2$,
$$
\partial^3_x h(x,s,u) = 2s(s-1)(2-s)x^{s-3} + (1+u)u(1-u) (x^{u-2} - x^{-u-2}) > 0.
$$
Hence $\partial^2_x h(x,s,u)$ is strictly increasing in $x$. Moreover $\partial_x h(1,s,u) = 2-2s-(1+u)-(1-u) =-2s < 0$ and $\partial_x h(x,s,u) > 0$ for large enough $x>1$. Thus $\partial_x h(\cdot,s,u)$ has a unique zero $x_0>1$, and $h(\cdot,s,u)$ is decreasing on $(1,x_0)$ and increasing on $(x_0,\infty)$. Since $h(1,s,u)=-4<0$ and $\lim_{x\to \infty} h(x,s,u)=+\infty$ this proves that the equation $h(x,s,u)=0$ has a unique real solution $x(s,u)> x_0>1$, depending continuously on $(s,u)$. Furthermore, $\partial_s h(x,s,u) < 0$ and $\partial_u h(x,s,u) \leq 0$, with equality if and only if $u=0$, hence by the implicit function theorem $x(s,u)$ is strictly increasing in $s$ and $u$ in $C_2$.
\end{proof}

In particular, the minimal largest root $\lambda$ is attained at~$a=1,\ b=0$, when $P(t)=t^{n}-2t^{\frac{n}{2}}-2t-1$. We can compare its root $\lambda$ to the largest real root $\alpha$ of $t^{n}-2t^{\frac{n}{2}}-1$. We find $P(\alpha)=\alpha^{n}-2\alpha^{\frac{n}{2}}-2\alpha-1 = -2\alpha < 0$ thus $\lambda > \alpha$. In particular, $\lambda^{n} > \alpha^n = 3+2\sqrt{2}$ as desired.

In case (3), if $P(t) = t^{n}-2t^{a}-2t^{b}-1$ is skew-reciprocal, then we reduce to case (1), so we can assume that it is skew-reciprocal up to cyclotomic factors but not skew-reciprocal. Suppose $P(t)$ has primitive $q^\text{th}$ root of unity $\xi=e^{i\theta}$ as a root, namely
\begin{equation}
\label{eq:roots:unity}
\xi^{n} = 2\xi^{a}+2\xi^{b}+1.
\end{equation}
Taking the norm of each side of this equation, we get
{\small \begin{align*}
1 &= (2\cos a\theta+2\cos b\theta+1)^2 + (2\sin a\theta+2\sin b\theta)^2 \\
&= 1+ 4(\cos^2 a\theta + \cos^2 b\theta + \cos a\theta + \cos b\theta + 2\cos a\theta\cos b\theta) \\
&~+ 4(\sin^2 a\theta + \sin^2 b\theta + 2\sin a\theta \sin b\theta) \\
&= 9 + 8(\cos a\theta\cos b\theta + \sin a\theta \sin b\theta) + 4\cos a\theta + 4\cos b\theta \\
&= 9 + 8 \cos (a-b)\theta + 8 \cos \frac{(a+b)\theta}{2} \cos \frac{(a-b)\theta}{2} \\
&= 1 + 16 \cos^2 \frac{(a-b)\theta}{2} + 8 \cos \frac{(a+b)\theta}{2} \cos \frac{(a-b)\theta}{2} \\
&= 1 + 8 \cos \frac{(a-b)\theta}{2} \left( 2 \cos \frac{(a-b)\theta}{2} + \cos \frac{(a+b)\theta}{2} \right) \\
&= 1 + 8 \cos \frac{(a-b)\theta}{2} \left( 2 \left( \cos \frac{a\theta}{2} \cos \frac{b\theta}{2} + \sin \frac{a\theta}{2} \sin \frac{b\theta}{2} \right) + \left( \cos \frac{a\theta}{2} \cos \frac{b\theta}{2} - \sin \frac{a\theta}{2} \sin \frac{b\theta}{2} \right) \right) \\
&= 1+ 8 \cos \frac{(a-b)\theta}{2} \left( 3 \cos \frac{a\theta}{2} \cos \frac{b\theta}{2} + \sin \frac{a\theta}{2} \sin \frac{b\theta}{2} \right)
\end{align*}}

Thus either $\cos \frac{(a-b)\theta}{2} = 0$ or $3 \cos \frac{a\theta}{2} \cos \frac{b\theta}{2} + \sin \frac{a\theta}{2} \sin \frac{b\theta}{2}=0$.

Since 
\begin{align*}
\xi^{a-b}+1 &= (1+\cos(a-b)\theta) + i\sin(a-b)\theta \\
&= 2\cos \frac{(a-b)\theta}{2} \left( \cos \frac{(a-b)\theta}{2} + i\sin \frac{(a-b)\theta}{2} \right),
\end{align*}
we have $\cos \frac{(a-b)\theta}{2} = 0$ if and only if $\xi^{a-b} = -1$ and $\xi^n = 1$.
This is in turn equivalent to $q$ being even, $a-b$ being an odd multiple of $\frac{q}{2}$, and $n$ being a multiple of $q$.
Also, in this case, $\frac{(a-b)\theta}{2}$ is an odd multiple of $\frac{\pi}{2}$ and we have $\tan \frac{a\theta}{2} \tan \frac{b\theta}{2} = -1$.

Meanwhile, $3 \cos \frac{a\theta}{2} \cos \frac{b\theta}{2} + \sin \frac{a\theta}{2} \sin \frac{b\theta}{2}=0$ if and only if
\begin{itemize}
    \item $\cos \frac{a\theta}{2} = 0$ and $\sin \frac{b\theta}{2} = 0$ or
    \item $\cos \frac{b\theta}{2} = 0$ and $\sin \frac{a\theta}{2} = 0$ or
    \item $\tan \frac{a\theta}{2} \tan \frac{b\theta}{2} = -3$.
\end{itemize}

In the first two cases, we have $\xi^{a-b} = -1$ and $\xi^n = 1$, which as pointed out above, is equivalent to $q$ being even, $a-b$ being an odd multiple of $\frac{q}{2}$, and $n$ being a multiple of $q$.

This reasoning shows that a root $\xi = e^{i\theta}$ of $P(t)$ that is a primitive $q^\text{th}$ root of unity is exactly one of the following two types:
\begin{enumerate}[label=(\roman*)]
    \item If $q$ is even, $a-b$ is an odd multiple of $\frac{q}{2}$, and $n$ is a multiple of $q$.
    \item If $\tan \frac{a\theta}{2} \tan \frac{b\theta}{2} = -3$.
\end{enumerate}

\begin{claim} \label{claim:5A1claim3}
Suppose there are no roots of type (ii). 
Then $P(t)$ cannot be skew-reciprocal up to cyclotomic factors. 
\end{claim}
\begin{proof}[Proof of \Cref{claim:5A1claim3}]
There is at least one root of type (i). In particular $n$ is even. Let $Q=\gcd(a-b,\frac{n}{2})$ and let $n=2NQ$. Without loss of generality suppose $a>b$ and let $a-b = AQ$.

We first compute the quotient $R(t)$ of $P(t)$ by $t^Q+1$:
\begin{align*}
P(t) &= t^{2NQ} - 2t^{AQ+b} - 2t^b - 1 \\
&= (t^{2NQ} - 1) - 2t^b (t^{AQ}+1).
\end{align*}
Hence 
$$R(t) = (t^{(2N-1)Q} - \cdots +t^Q -1) - 2t^b (t^{(A-1)Q} - \cdots -t^Q+1).$$
We then check that $R(t)$ has no root of unity as a root. The only possible roots satisfy $\xi^Q = \pm 1$, but
\begin{align*}
R(\xi) &= \begin{cases} -2N - 2\xi^bA & \text{if $\xi^Q=-1$} \\ -2\xi^b & \text{if $\xi^Q=1$} \end{cases} \\
&\neq 0.
\end{align*}
Here we used the fact that $\gcd(A,N)=1$ to see that the top expression is nonzero.

Finally, we check that $R(t)$ is not skew-reciprocal. If $Q >1$, then due to primitivity, $b \not\equiv 0$ (mod $Q)$ and the coefficients of the two terms occupy disjoint sets of degrees, so $R(t)$ cannot be skew-reciprocal unless $Q$ is even and $\frac{(A-1)Q}{2}+b = \frac{(2N-1)Q}{2}$. But in this case we have $a+b = n$ and we reduce to case (1) tackled above. If $Q=1$, then $R(t)$ has odd degree and cannot be skew-reciprocal.
\end{proof}

Hence we can suppose that there is a root $\xi = e^{i\theta}$ of type (ii). Say $\xi$ is a primitive $q^\text{th}$ root of unity, then we have $\tan \frac{a\pi}{q} \tan \frac{b\pi}{q} = -3$.
If $q$ were odd, then using Galois conjugation, we find another root of unity $\xi^2$ as a root of $P(t)$. Noting that $\xi^2$ must be of type (ii) as well, this leads to another equation
\begin{equation}
\label{eq:new:eq}
\tan \frac{2a\pi}{q} \tan \frac{2b\pi}{q} = -3.
\end{equation}

\begin{claim} \label{claim:5A1claim4}
If (\ref{eq:new:eq}) is satisfied for every type (ii) root, then $P(t)$ cannot be skew-reciprocal up to cyclotomic factors. 
\end{claim}

\begin{proof}[Proof of \Cref{claim:5A1claim4}]
Let us write $\alpha = \tan \frac{a\pi}{q}$ and $\beta=\tan \frac{b\pi}{q}$. Then we have $\alpha\beta=-3$ and $\frac{2\alpha}{1-\alpha^2}\frac{2\beta}{1-\beta^2} = -3$. Hence
$$4 = (1-\alpha^2)(1-\beta^{2}) = 1- (\alpha + \beta)^2 + 2\alpha  \beta + (\alpha\beta)^2 = 1 - (\alpha + \beta)^2 - 6 + 9.$$
This leads to $(\alpha + \beta)^2 = 0$, so $\alpha = - \beta = \pm \sqrt{3}$. Without loss of generality suppose $\tan \frac{a\pi}{q} = \alpha = \sqrt{3}$ and $\tan \frac{b\pi}{q} = \beta = -\sqrt{3}$. Then $\xi^a = e^{\frac{2ia\pi}{q}} = e^{\frac{2i\pi}{3}}$ and $\xi^b = e^{\frac{2ib\pi}{q}} = e^{\frac{4i\pi}{3}}$, and $\xi^n = 2\xi^a+2\xi^b+1 = -1$.

Since $\xi^{3a} = \xi^{3b} = \xi^{2n} = 1$, $q|\gcd(3a,3b,2n) \leq 6 \gcd (a,b,n) = 6$. In fact, since $\xi^a$ and $\xi^b$ are third roots of unity while $\xi^n$ is a second root of unity, we have $q=6$. In turn, $a=6A+2$, $b=6B+4$, and $n=6N+3$. In particular, since $n$ is odd, $P(t)$ cannot have any roots of type (i).

We first compute the quotient $R(t)$ of $P(t)$ by $t^2-t+1$:
\begin{align*}
P(t) &= t^{6N+3}-2t^{6A+2}-2t^{6B+4}-1 \\
&= (t^{6N}-1)t^3 - 2(t^{6A}-1)t^2 - 2(t^{6B}-1)t^4 - (2t^4-t^3+2t^2+1)
\end{align*}

Hence
\begin{align*}
R(t) &= (t^{6N-3}- \cdots +t^3-1)(t+1)t^3 - 2(t^{6A-3}- \cdots +t^3-1)(t+1)t^2 \\
&~- 2(t^{6B-3}- \cdots +t^3-1)(t+1)t^4 - (2t^2+t+1)
\end{align*}

We then check that $R(t)$ has no root of unity as a root. By the reasoning above, the only possible roots $\xi$ satisfy $\xi^3=-1$, but
\begin{align*}
R(\xi) &= 2N(\xi+1)-4A(\xi+1)\xi^2-4B(\xi+1)\xi^4-(2\xi^2+\xi+1) \\
&= 2N(\xi+1) - 4A(\xi-2) - 4B(-2\xi+1)-(3\xi-1) \\
&= (2N-4A+8B-3)\xi + (2N+8A-4B+1) \neq 0
\end{align*}
since $N,A,B$ are integers.

Finally, since $R(t)$ is of odd degree, it cannot be skew-reciprocal.
\end{proof}

Thus $q$ is even for at least some type (ii) root $\xi=e^{\frac{2i\pi}{q}}$. For this value of $q$, we set $q=2^m u$ with $u$ odd and $m \geq 1$. Since $u-2$ and $q=2^m u$ are coprime, $P(t)$ has the root of unity $\xi^{u-2}$ as a root, necessarily of type (ii).
We thus have a new relation $\tan a(u-2)\frac{\pi}{q} \tan b(u-2)\frac{\pi}{q} = -3$.

An elementary calculation gives $\tan a(u-2)\frac{\pi}{q} = \tan \left( \frac{a\pi}{2^m} - \frac{2a\pi}{q} \right)$ and $\tan b(u-2)\frac{\pi}{q} = \tan \left( \frac{b\pi}{2^m} - \frac{2b\pi}{q} \right)$. If $m=1$ and $a,b$ are even, one has $\tan \left( \frac{a\pi}{2^m} - \frac{2a\pi}{q} \right) = -\tan \frac{2a\pi}{q}$ and $\tan \left( \frac{b\pi}{2^m} - \frac{2b\pi}{q} \right) = -\tan \frac{2b\pi}{q}$. If this is the case for every type (ii) root with an even value of $q$, then we are reduced to the previous situation~\eqref{eq:new:eq} and we again run into a contradiction.
Hence we can assume that either $m=1$ and at least one of $a,b$ is odd, or $m>1$.

If $m=1$, $a$ is odd, and $b$ is even, one has $\tan \left( \frac{a\pi}{2^m} - \frac{2a\pi}{q} \right) = \frac{1}{\tan \frac{2a\pi}{q}}$
This gives
$$\frac{-\tan \frac{2b\pi}{q}}{\tan \frac{2a\pi}{q}} = -3,$$
or equivalently 
$$3\frac{2\alpha}{1-\alpha^2}=\frac{2\beta}{1-\beta^2}=\frac{-\frac{6}{\alpha}}{1-\frac{9}{\alpha^2}} = \frac{-6\alpha}{\alpha^2-9}$$
$$\frac{\alpha^2-9}{\alpha^2-1} = 1$$
which is absurd. Symmetrically, we also get a contradiction if $m=1$, $a$ is even, and $b$ is odd.

If $m=1$ and $a,b$ are odd, we have
$$\tan \frac{2a\pi}{q} \tan \frac{2b\pi}{q} = -\frac{1}{3},$$
or equivalently 
$$-3\frac{2\alpha}{1-\alpha^2}= \frac{1-\beta^2}{2\beta} = \frac{1-\frac{9}{\alpha^2}}{-\frac{6}{\alpha}} = \frac{\alpha^2-9}{-6\alpha}$$
$$36\alpha^2=(\alpha^2-9)(1-\alpha^2)$$
$$\alpha^4 + 26 \alpha^2 + 9 = 0$$
which is also absurd since $\alpha$ is a real number. 

Hence we can assume $m > 1$.
In this case we will arrive at a contradiction by considering the number fields $\QQ(e^{\frac{2i\pi}{2^{m}u}})$ and $\QQ(e^{\frac{2i\pi}{2^{m-1}u}})$. The degrees of these number fields are, respectively, $\varphi(2^{m}u)=2^{m-1}\varphi(u)$ and $\varphi(2^{m-1}u)=2^{m-2}\varphi(u)$, hence the latter is strictly contained in the former.

Now by primitiveness, at least one of $n,a,b$ is odd. By grouping up the terms in (\ref{eq:roots:unity}) involving an odd power of $\xi$, we would be able to express $\xi = e^{\frac{2i\pi}{2^{m}u}}$ as an element of $\QQ(e^{\frac{2i\pi}{2^{m-1}u}})$, leaving us with a contradiction, unless if $n$ is even, $a$ and $b$ are odd, and $\xi^n - 1 = 0$ and $\xi^a+\xi^b = 0$. But then $\xi$ would be of type (i), contradicting what we have assumed above.

\subsection*{Case $A^*_2$.} In this case~$Q(t)=1-t^a-t^b-t^c+t^{a+b},$ where~$c$ is the weight on the isolated vertex. As in case $2A_1$, $n$ must be even. Let us write $n=2g$. There are two cases to consider: either~$a+b = 2g$ or~$c=2g$. 

In the former case, the characteristic polynomial is of the form~$P(t)=t^{2g}-t^{2g-a}-t^c-t^{a}+1.$ The only way to satisfy the parity condition of \Cref{lemma:parity} is if~$c=g$. In this case, the polynomial is actually reciprocal. In particular, the normalized spectral radius of the matrix~$A$ is bounded from below by~$\mu^4 > 3+2\sqrt{2}$ by McMullen's analysis~\cite[Section~7]{McM15}. 

In the latter case, the characteristic polynomial is of the form~$P(t)=t^{2g}-t^a-t^b+t^{a+b-2g}-1.$ There are three coefficients, so to satisfy the parity condition of \Cref{lemma:parity} at least one of them must be the middle coefficient: either~$a=g$,~$b=g$ or~$a+b=3g$.
\begin{enumerate}
    \item[(i)] If~$a=g$, then we get characteristic polynomials of the form~$P(t)=t^{2g}-t^b-t^g+t^{b-g}-1.$ In order to satisfy the parity condition of \Cref{lemma:parity}, we must have~$(b-g) +b = 2g$, that is, $g$ is even and $b=\frac{3g}{2}$. If $g \geq 4$ then $P(t)$ is a polynomial in~$t^\frac{g}{2},$ hence cannot be the characteristic polynomial of a primitive matrix. If $g=2$ then one can check that $P(t) = t^4-t^3-t^2+t-1$ is not skew-reciprocal up to cyclotomic factors directly.
    \item[(ii)] If~$b=g$, we can repeat the same argument as for~$a=g$.
    \item[(iii)] If~$a+b=3g$, then we get characteristic polynomials of the form~$P(t)=t^{2g}-t^a-t^b+t^g-1$. The only way to satisfy the parity condition of \Cref{lemma:parity} is if~$a=b$, i.e. if $g$ is even and $a=b=\frac{3g}{2}$. If $g \geq 4$ then $P(t)$ is a polynomial in~$t^\frac{g}{2},$ hence cannot be the characteristic polynomial of a primitive matrix. If $g=2$ then $P(t) = t^4-2t^3+t^2-1=(t^2-t-1)(t^2-t+1)$, thus $\rho(A)^n = \mu^4 > 3+2\sqrt{2}$. 
\end{enumerate}

The proof of Theorem~\ref{thm:skucspectralradius} is now complete.
\qed

As mentioned below the statement of \Cref{thm:skucspectralradius}, \Cref{thm:standardlyembeddedtt}, \Cref{prop:orirevpamaptononcycloskewreciprocal}, and \Cref{thm:skucspectralradius} together imply the inequality in \Cref{thm:mainthm}.

\section{Sharpness of \Cref{thm:mainthm}}
\label{sec:sharpness}

In this section, we demonstrate the asymptotic sharpness of \Cref{thm:mainthm}. 

Given an integer $k \geq 2$, we let $S_k$ be the surface 
$$\left( S^1 \times (-1,1) \right) \backslash \left\{ (\exp{\frac{ai \pi}{k}},0) \mid a \in \mathbb{Z}/2k \right\}.$$
That is, we puncture $2k$ points from an open annulus in a cyclic fashion. See \Cref{fig:sharpnesseg}. Note that $\chi(S_k)=-2k$.

\begin{figure}
    \centering
    \fontsize{6pt}{6pt}\selectfont
    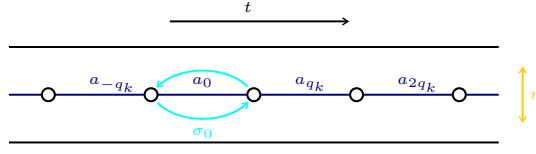
    \caption{The surface $S_k$ and several maps on $S_k$.}
    \label{fig:sharpnesseg}
\end{figure}

We define
$$p_k = \begin{cases}
k+1 & \text{if $k$ is even} \\
k+2 & \text{if $k$ is odd}
\end{cases}$$
Note that $p_k$ and $2k$ are relatively prime, hence we can set $q_k$ to be an inverse of $p_k \mod 2k$.

We set up some maps on $S_k$ which we will use to define $f_k$: We denote by $a_n$ be the arc $\left\{ (\exp{i\theta},0) \mid \theta \in (\frac{nq_k \pi}{k},\frac{(nq_k+1) \pi}{k}) \right\}$, and denote by $\sigma_n$ the positive half-twist around $a_n$. We denote by $t$ the translation by $q_k$ units, i.e. $t(\exp{i\theta}, r) = (\exp{i(\theta+\frac{q_k \pi}{k})},r)$. Note that $t$ maps each arc $a_n$ to $a_{n+1}$. Finally, we denote by $\eta$ the reflection across $S^1 \times \{0\}$, i.e. $\eta(\exp{i\theta}, r) = (\exp{i\theta}, -r)$. See \Cref{fig:sharpnesseg}.

We then define $f_k$ to be the composition $\eta t \sigma_0$.

\begin{prop} \label{prop:sharpnesseg}
The map $f_k$ is an orientation-reversing fully-punctured pseudo-Anosov map with at least two puncture orbits. The stretch factor of $f_k$ is the largest root of
$$\begin{cases}
t^{2k}-t^{k+1}-t^{k-1}-1 & \text{if $k$ is even} \\
t^{2k}-t^{k+2}-t^{k-2}-1 & \text{if $k$ is odd.} 
\end{cases}$$
\end{prop}
\begin{proof}
It is clear that $f_k$ is orientation-reversing and has at least two puncture orbits. To show the rest of the properties, we will pass to a double branched cover.

Consider the double cover of $S_k$ determined by sending loops around each puncture $(\exp{\frac{a i\pi}{k}},0)$ to $1 \in \mathbb{Z}/2$ and sending the loops $S^1 \times \{r\}$ to $0 \in \mathbb{Z}/2$. We let $\widehat{S_k}$ be the surface obtained by filling in the punctures lying above each $(\exp{\frac{a i\pi}{k}},0)$. Note that each arc $a_n$ lifts to a closed curve $c_n$ in $\widehat{S_k}$. It is convenient to think of $\widehat{S_k}$ as a chain of $2k$ bands with cores $c_n$. See \Cref{fig:sharpnessdoublecover}. Note that for each $n$, $c_n$ intersects each of $c_{n+p_k}$ and $c_{n-p_k}$ once, and intersects those curves only.

Each positive half-twist $\sigma_n$ on $S_k$ lifts to a positive Dehn twist $\widehat{\sigma_n}$ around $c_n$. The translation $t$ lifts to the translation $\widehat{t}$ that sends each curve $c_n$ to $c_{n+1}$. The reflection $\eta$ lifts to a reflection $\widehat{\eta}$ which preserves each $c_n$. Thus $f_k$ lifts to the map $\widehat{f_k} = \widehat{\eta} \widehat{t} \widehat{\sigma_0}$. It suffices to show that $\widehat{f_k}$ is a fully-punctured pseudo-Anosov map with stretch factor as recorded in \Cref{prop:sharpnesseg}.

Our proof for this is a copy of the arguments in \cite{LS20}. In fact, for $k$ even, it is easy to see that $\widehat{f_k}$ is exactly the map denoted by $\psi_k$ in \cite[Section 3]{LS20}.

Let $\tau_k$ be the train track obtained by taking the union of the $c_n$ and smoothing according to their orientations. Note that $\tau_k$ is a deformation retract of $\widehat{S_k}$. Also, note that each $c_n$ is carried by $\tau_k$, that is, each of $c_n$ determines an element of the weight space $\mathcal{W}(\tau_k)$, which we denote by $c_n$ as well. Observe that the set $\{c_n\}$ is linearly independent, and $(f_k)_*$ preserves the subspace $U$ spanned by $\{c_n\}$. In fact, under the basis $(c_1,...,c_{2k})$, the action of $(f_k)_*$ is represented by the matrix $P+N$ where
$$P_{ij} = 
\begin{cases}
1 & \text{if $i=j+1$} \\
0 & \text{otherwise}
\end{cases}$$
and
$$N_{ij} = 
\begin{cases}
1 & \text{if $(i,j) = (1,p_k)$ or $(1,-p_k)$} \\
0 & \text{otherwise}
\end{cases}$$
where we regard the values of the indices mod $2k$.

\begin{figure}
    \centering
\begingroup%
  \makeatletter%
  \providecommand\color[2][]{%
    \errmessage{(Inkscape) Color is used for the text in Inkscape, but the package 'color.sty' is not loaded}%
    \renewcommand\color[2][]{}%
  }%
  \providecommand\transparent[1]{%
    \errmessage{(Inkscape) Transparency is used (non-zero) for the text in Inkscape, but the package 'transparent.sty' is not loaded}%
    \renewcommand\transparent[1]{}%
  }%
  \providecommand\rotatebox[2]{#2}%
  \newcommand*\fsize{\dimexpr\f@size pt\relax}%
  \newcommand*\lineheight[1]{\fontsize{\fsize}{#1\fsize}\selectfont}%
  \ifx\svgwidth\undefined%
    \setlength{\unitlength}{162.46609065bp}%
    \ifx\svgscale\undefined%
      \relax%
    \else%
      \setlength{\unitlength}{\unitlength * \real{\svgscale}}%
    \fi%
  \else%
    \setlength{\unitlength}{\svgwidth}%
  \fi%
  \global\let\svgwidth\undefined%
  \global\let\svgscale\undefined%
  \makeatother%
  \begin{picture}(1,0.80777953)%
    \lineheight{1}%
    \setlength\tabcolsep{0pt}%
    \put(0,0){\includegraphics[width=\unitlength,page=1]{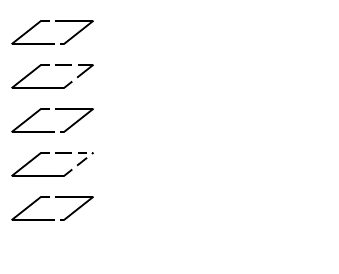}}%
    \put(0.12218852,0.00706884){\color[rgb]{0,0,0}\makebox(0,0)[lt]{\lineheight{1.25}\smash{\begin{tabular}[t]{l}$\widehat{S}_k$\end{tabular}}}}%
    \put(0.56337693,0.00706098){\color[rgb]{0,0,0}\makebox(0,0)[lt]{\lineheight{1.25}\smash{\begin{tabular}[t]{l}$S_k$\end{tabular}}}}%
    \put(0,0){\includegraphics[width=\unitlength,page=2]{sharpnessdoublecover.pdf}}%
    \put(0.32554457,0.38543335){\color[rgb]{0,0,0.50196078}\makebox(0,0)[lt]{\lineheight{1.25}\smash{\begin{tabular}[t]{l}$c_0$\end{tabular}}}}%
    \put(0.7729336,0.38543861){\color[rgb]{0,0,0.50196078}\makebox(0,0)[lt]{\lineheight{1.25}\smash{\begin{tabular}[t]{l}$a_0$\end{tabular}}}}%
    \put(0,0){\includegraphics[width=\unitlength,page=3]{sharpnessdoublecover.pdf}}%
  \end{picture}%
\endgroup%

    \caption{The double branched cover $\widehat{S_k}$ of $S_k$.}
    \label{fig:sharpnessdoublecover}
\end{figure}

It is straightforward to check that $P+N$ is primitive, thus $\widehat{f_k}$ is a fully-punctured pseudo-Anosov map and $\tau_k$ fully carries the unstable foliation of $\widehat{f_k}$.

Using the fact that $P+N$ is a companion matrix, one computes the characteristic polynomial of $P+N$ to be 
$$t^{2k}-t^{p_k} - t^{2k-p_k}-1 = \begin{cases}
t^{2k}-t^{k+1}-t^{k-1}-1 & \text{if $k$ is even} \\
t^{2k}-t^{k+2}-t^{k-2}-1 & \text{if $k$ is odd.}
\end{cases}$$
Thus the spectral radius of $P+N$ is the largest root of this polynomial. Applying the Perron-Frobenius theorem we conclude that the stretch factor of $\widehat{f_k}$ is this largest root as well.
\end{proof}

The normalized stretch factor $P_k = \lambda(f_k)^{2k}$ of $f_k$ is in turn the largest root of the equation
$$\begin{cases}
P_k-(P_k^{\frac{1}{2k}}+P_k^{-\frac{1}{2k}})P_k^{\frac{1}{2}}-1 = 0 & \text{if $k$ is even} \\
P_k-(P_k^{\frac{1}{k}}+P_k^{-\frac{1}{k}})P_k^{\frac{1}{2}}-1 = 0 & \text{if $k$ is odd.} 
\end{cases}$$
As $k \to \infty$, both families of equations converge to $P-2P^{\frac{1}{2}}-1=0$, whose largest root is $3+2\sqrt{2}$. Thus we conclude that $\lim_{k \to \infty} P_k = 3+2\sqrt{2}$.

\section{Proof of \Cref{thm:set}} \label{sec:thmsetproof}

We explain the proof of \Cref{thm:set} in this section. The strategy is much simpler than that of \Cref{thm:mainthm} and essentially consists of determining, among the orientation-preserving fully-punctured pseudo-Anosov maps that have normalized stretch factor strictly less than~$\mu^4$, which ones admit an orientation-reversing square root. 

The set of normalized stretch factors of such maps is determined by Tsang~\cite[Table 1]{Tsa23} and consists of 
\begin{enumerate}
\item four quadratic stretch factors, $\mu^2,\frac{4+\sqrt{12}}{2},\frac{5+\sqrt{21}}{2},\sigma^2$, given by linear Anosov homeomorphisms on the punctured torus $S_{1,1}$ with Euler characteristic $-1$,
\item two normalized stretch factors $(\text{Lehmer’s number})^9 \approx 4.311$, realized on $S_{5,1}$, and $|\mathrm{LT_{1,2}}|^3 \approx 5.107$, realized  on $S_{2,1}$.
\end{enumerate}
We immediately see that $\mu$ and $\sigma$ are normalized stretch factors of fully-punctured orientation-reversing pseudo-Anosov maps, given, respectively, by matrices $\left(\begin{smallmatrix} 1 & 1 \\ 1 & 0 \end{smallmatrix} \right),\left(\begin{smallmatrix} 2 & 1 \\ 1 & 0 \end{smallmatrix} \right)\in \mathrm{GL}_2(\ZZ)$ on $S_{1,1}$.
Meanwhile, the normalized stretch factor $\mu^2$ is realized on the punctured sphere $S_{0,4}$ (see \Cref{rmk:2A_1lowdeg}).

For the other two  quadratic stretch factors $\frac{4+\sqrt{12}}{2}$ and $\frac{5+\sqrt{21}}{2}$, one can check that the minimal polynomial of their square roots are, respectively,
$x^4 - 4x^2 + 1$ and $x^4 - 5x^2 + 1$. Since the degree is not $2$, they cannot be the normalized stretch factor of an orientation-reversing Anosov map on $S_{1,1}$.

On the other hand, Lehmer’s number, respectively $|\mathrm{LT_{1,2}}|$, is the largest real root of the polynomial $t^{10} +t^9 -t^7 -t^6 -t^5 -t^4 -t^3 +t+1$, respectively $t^4-t^3-t^2-t+1$.
They are in particular Salem numbers (all roots different from $\lambda^{\pm 1}$ have modulus one). By a result of Strenner and Liechti~\cite[Theorem 1.10]{LS20}, the Galois conjugates of the stretch factor of an orientation-reversing pseudo-Anosov map cannot lie on the unit circle.

Finally, the fact that the spectrum contains a dense subset of $[\sigma^2,\infty)$ follows from the examples in \Cref{sec:sharpness} (see the discussion in~\cite[Section 6.1]{Tsa23}).

\section{Discussion and further questions}
\label{sec:questions}

As mentioned in the introduction, \Cref{thm:mainthm} is motivated by \cite{HT22}. However, the statement of \Cref{thm:mainthm} is less satisfying than the results in \cite{HT22} in two aspects:
\begin{itemize}
    \item \cite{HT22} determines the actual value of the minimal normalized stretch factor on surfaces with Euler characteristic $-2k$, for each $k$. Meanwhile, \Cref{thm:mainthm} only determines the asymptotic value of the minimal values in the orientation-reversing case, as $k \to \infty$. 
    \item \cite{HT22} shows that the minimal normalized stretch factors on surfaces with odd Euler characteristic are greater than those on surfaces with even Euler characteristic. It is conceivable that this remains true in the orientation-reversing case, especially since we can only prove that \Cref{thm:mainthm} is sharp for surfaces with even Euler characteristic, but for now we do not know for sure.
\end{itemize}

Regarding the first point, we conjecture the following.

\begin{conj} \label{conj:exactvalue}
Let~$f:S \to S$ be an orientation-reversing fully-punctured pseudo-Anosov map on a finite-type orientable surface. Suppose $-\chi(S) = 2k \geq 4$ and $f$ has at least two puncture orbits. Then the normalized stretch factor of~$f$ is greater than or equal to the largest real root of 
$$\begin{cases}
t^{2k}-t^{k+1}-t^{k-1}-1 & \text{if $k$ is even} \\
t^{2k}-t^{k+2}-t^{k-2}-1 & \text{if $k$ is odd.}
\end{cases}$$
\end{conj}

In other words, we conjecture that the examples we demonstrated in \Cref{sec:sharpness} attain the minimal normalized stretch factor among orientation-reversing fully-punctured pseudo-Anosov maps defined on surfaces with even Euler characteristic and with at least two puncture orbits.

Theoretically, one should be able to verify \Cref{conj:exactvalue}, and, separately, address the second point above, by following the approach in this paper. The additional work lies in extending the analysis in \Cref{section:curvegraphanalysis} to a larger collection of curve graphs. 

The main deterrent for doing so is the anticipated complexity of the analysis. We currently do not have many tools for identifying polynomials that are skew-reciprocal up to cyclotomic factors. Our parity condition \Cref{lemma:parity} is a rough necessary condition that helps reduce the analysis into finitely many families, but the number of parameters of these families grows as the number of terms in the characteristic polynomial increases. The characteristic polynomial in case $5A_1$ has 6 terms and the analysis is already quite involved. For the curve graphs $A^{**}_2$ and $A^*_3$, which have the next largest minimal growth rate, the characteristic polynomial will have $6$ and $7$ terms respectively, and if one repeats our strategy on these cases the analysis is expected to be even more unwieldy. 

Worse still, for $k=3$, one would have to extend the analysis to all curve graphs with growth rate $<8.186$. We currently do not have a list of such graphs, since McMullen \cite{McM15} only compiled all graphs with minimum growth rate $\leq 8$.
Note that this issue also arises in \cite{Lie23}, where additional techniques had to be employed to compute the minimal spectral radius for primitive skew-reciprocal $6$-by-$6$ matrices.

In conclusion, it would be wise to develop a deeper understanding of polynomials that are skew-reciprocal up to cyclotomic factors before attempting to generalize our analysis by brute force.

We now turn to address \Cref{thm:set}. An obvious future direction regarding this theorem is to understand the behavior of normalized stretch factors in the range $(\mu^2, \sigma^2)$. We make the following conjecture.

\begin{conj} \label{conj:minaccum}
The minimal accumulation point of the set of normalized stretch factors of orientation-reversing fully-punctured pseudo-Anosov maps is $\sigma^2$.
\end{conj}

In other words, there are only finitely many isolated values for the normalized stretch factor in the range $(\mu^2, \sigma^2)$.

Note that if \Cref{conj:minaccum} is true, then it would imply that \Cref{thm:mainthm} and \Cref{cor:maincor} remain true without the `with at least two puncture orbits' hypothesis.

Our current approach is inadequate towards verifying this conjecture because there are infinitely many fully-punctured orientation-preserving pseudo-Anosov maps with normalized stretch factor in the range $(\mu^4, \sigma^4)$. It would take an infinite amount of time to check whether each of these admit an orientation-reversing square root, at least without some new ideas.

Another approach might be to directly generalize the ideas in \cite{Tsa23} to the orientation-reversing case. In \cite{Tsa23}, it is shown that the mapping torus of an orientation-preserving fully-punctured pseudo-Anosov map with normalized stretch factor $\leq \mu^4$ admits a \textit{veering triangulation} with an explicitly bounded number of tetrahedra. Now, veering triangulations can only exist on orientable 3-manifolds, but there is a generalization of the notion of veering triangulations to \textit{veering branched surfaces}, and the latter can exist on non-orientable 3-manifolds. We refer to \cite{Tsa22} for details.

One should be able to generalize the ideas in \cite{Tsa23} to show that the mapping torus of a fully-punctured orientation-reversing pseudo-Anosov map with normalized stretch factor $\leq \sigma^2$ admits a veering branched surface with an explicitly bounded number of triple points. If this number is reasonably small, one could then attempt to generate a list of all veering branched surfaces up to that number of triple points and study the maps that arise as monodromies of fiberings of the corresponding 3-manifolds.

\bibliographystyle{alphaurl}

\bibliography{bib.bib}

\begin{thebibliography}{McM15}

\bibitem[Dim19]{Dim19}
Vesselin Dimitrov.
\newblock A proof of the {S}chinzel-{Z}assenhaus conjecture on polynomials, 2019.
\newblock \href {https://arxiv.org/abs/1912.12545} {\path{arXiv:1912.12545}}.

\bibitem[FLP12]{FLP79}
Albert Fathi, Fran\c{c}ois Laudenbach, and Valentin Po\'{e}naru.
\newblock {\em Thurston's work on surfaces}, volume~48 of {\em Mathematical Notes}.
\newblock Princeton University Press, Princeton, NJ, 2012.
\newblock Translated from the 1979 French original by Djun M. Kim and Dan Margalit.

\bibitem[HT22]{HT22}
Eriko Hironaka and Chi~Cheuk Tsang.
\newblock Standardly embedded train tracks and pseudo-{A}nosov maps with minimum expansion factor, 2022.
\newblock \href {https://arxiv.org/abs/2210.13418} {\path{arXiv:2210.13418}}.

\bibitem[Lie19]{Lie19}
Livio Liechti.
\newblock On the arithmetic and the geometry of skew-reciprocal polynomials.
\newblock {\em Proc. Amer. Math. Soc.}, 147(12):5131--5139, 2019.
\newblock \href {https://doi.org/10.1090/proc/14668} {\path{doi:10.1090/proc/14668}}.

\bibitem[Lie24]{Lie23}
Livio Liechti.
\newblock On the minimal spectral radii of skew-reciprocal integer matrices.
\newblock {\em New York J. Math.}, 30:307--322, 2024.

\bibitem[LS20]{LS20}
Livio Liechti and Bal\'{a}zs Strenner.
\newblock Minimal pseudo-{A}nosov stretch factors on nonoriented surfaces.
\newblock {\em Algebr. Geom. Topol.}, 20(1):451--485, 2020.
\newblock \href {https://doi.org/10.2140/agt.2020.20.451} {\path{doi:10.2140/agt.2020.20.451}}.

\bibitem[McM15]{McM15}
Curtis~T. McMullen.
\newblock Entropy and the clique polynomial.
\newblock {\em J. Topol.}, 8(1):184--212, 2015.
\newblock \href {https://doi.org/10.1112/jtopol/jtu022} {\path{doi:10.1112/jtopol/jtu022}}.

\bibitem[SZ65]{SZ65}
A.~Schinzel and H.~Zassenhaus.
\newblock A refinement of two theorems of {K}ronecker.
\newblock {\em Michigan Math. J.}, 12:81--85, 1965.
\newblock URL: \url{http://projecteuclid.org/euclid.mmj/1028999247}.

\bibitem[Tsa23a]{Tsa23}
Chi~Cheuk Tsang.
\newblock On the set of normalized dilatations of fully-punctured pseudo-{A}nosov maps, 2023.
\newblock \href {https://arxiv.org/abs/2306.10245} {\path{arXiv:2306.10245}}.

\bibitem[Tsa23b]{Tsa22}
Chi~Cheuk Tsang.
\newblock Veering branched surfaces, surgeries, and geodesic flows.
\newblock {\em New York J. Math.}, 29:1425--1495, 2023.

\end{thebibliography}

\end{document}